%% file: wave-boundary-limit-arxiv.tex
\documentclass[12pt,a4paper]{amsart}%
\AtBeginDocument{{\noindent\small
\emph{Preprint submitted to arXiv
\hfill \today}}
\vspace{8mm}}

\input{body-pages/preambule}

\title[Boundary null controllability]{
Boundary null controllability of degenerate wave equation as the limit of internal controllability
 }

\author[Ara\'ujo]{Bruno S. V. Ara\'ujo}
\address[B. S\'ergio Ara\'ujo]{\newline 
Unidade Acadêmica de Matemática, Universidade Federal de Campina Grande, Campina Grande, PB, Brazil
}
\email{bsergio@mat.ufcg.edu.br}

\author[R. Demarque]{Reginaldo Demarque}
\address[R. Demarque]{\newline Departamento de Ciências da Natureza,
	Universidade Federal Fluminense,
	Rio das Ostras, RJ, 28895-532, Brazil }
\email{reginaldodr@id.uff.br}

\author[L. Viana]{Luiz Viana}
\address[L. Viana]{\newline Departamento de Análise,
	Universidade Federal Fluminense,
	Niter\'{o}i, RJ, Brazil }
\email{luizviana@id.uff.br}


\begin{document}

\begin{abstract}
This work is concerned with the possibility of proving the boundary null controllability for the degenerate wave equation, developing the asymptotic analysis of a suitable family of {state-control pairs} $((u_\ep , v_\ep))_{\ep >0}$, solving related internal null controllability problems. The passage to the limit argument will be rigorously performed through the obtainment of a refined observability type inequality, with a constant explicitly given in terms of $\ep >0$. This represents an essential point, since will allow us to achieve our required weak convergence results. 
\end{abstract}
\maketitle

\onehalfspacing
\input{body-pages/introduction}

\input{body-pages/preliminaries}

\input{body-pages/sec3}

\input{body-pages/sec4}
\input{body-pages/sec5}
\input{body-pages/sec6}
\input{body-pages/sec7}

\appendix
\renewcommand{\thesection}{\Alph{section}}
\input{body-pages/apd1}
\input{body-pages/apd2}
\input{body-pages/apd4}


\newpage
\bibliographystyle{acm}
\bibliography{references}
\end{document}

%% file: body-pages/preambule.tex
\usepackage{fullpage}
\usepackage[utf8]{inputenc}
\usepackage[pdftex,colorlinks]{hyperref}
\usepackage{dsfont}
\usepackage{amsmath,amsthm,amsfonts,amssymb}
\usepackage{graphicx}%
\usepackage{enumerate}
\usepackage{color}
\usepackage{tcolorbox}
\usepackage{soul}
\usepackage{soulpos}
\usepackage{tikz}
\usetikzlibrary{patterns}
\usepackage{mathtools}
\usepackage{setspace}

\newtheorem{thm}{Theorem}[section]
\newtheorem{cor}[thm]{Corollary}
\newtheorem{lem}[thm]{Lemma}
\newtheorem{prop}[thm]{Proposition}
\theoremstyle{definition}
\newtheorem{defn}[thm]{Definition}
\theoremstyle{remark}
\newtheorem{rem}[thm]{Remark}

\numberwithin{equation}{section}

\makeatother

\soulregister{\ha}{0}

\usepackage{tikz}
\usetikzlibrary{shapes,decorations}

\usepackage{environ}
\definecolor{mycolor}{RGB}{255,251,204}
\tikzstyle{mybox} = [draw=yellow, very thick,
rectangle, rounded corners, inner sep=5pt, inner ysep=10pt]
\tikzstyle{fancytitle} =[fill=white, text=red]

\NewEnviron{blc}{\begin{tikzpicture}
\node [mybox] (box){%
\begin{minipage}{0.9\textwidth}
\BODY
\end{minipage}
};
\end{tikzpicture}
}

\newcommand{\dps}{\displaystyle}
\newcommand{\ep}{\varepsilon} 
\newcommand{\R}{\mathbb{R}}	  
\newcommand{\D}{\mathcal{D}}
\newcommand{\n}[2]{\|#1\|_{{#2}}} 

\newcommand{\dd}{\,dx\,dt} 
\newcommand{\into}{\int_0^1}	
\newcommand{\intq}{\int_0^T\int_0^1}	
\newcommand{\intw}{\int_0^T\int_{1-\ep}^1}	
\newcommand{\domw}{Q_{\ep}}

\newcommand{\ac}{AC}
\newcommand{\acloc}{AC_{\mathrm{loc}}}
\newcommand{\ha}{H^1_\alpha}
\newcommand{\haa}{H^2_\alpha}
\newcommand{\had}{H^{-1}_\alpha}
\newcommand{\habd}{H^{-2}_\alpha}
\newcommand{\nhd}[1]{\n{#1}{H^{-1}_\alpha}}
\newcommand{\nl}[1]{\n{#1}{L^2(0,1)}}
\newcommand{\nh}[1]{\n{#1}{H^1_\alpha}}

\newcommand{\dpa}[2]{\left\langle #1,#2\right\rangle_{_{\mkern-5mu \had, \ha}}}
\newcommand{\dpaa}[2]{\left\langle #1,#2\right\rangle_{_{\mkern-5mu\habd, \haa}}}
\newcommand{\pl}[2]{\left(#1,#2\right)_{L^2(0,1)}}
\newcommand{\iphad}[2]{\left( #1,#2\right)_{_{\mkern-5mu \had}}}

\newcommand{\A}{A}
\newcommand{\phin}{\phi_{n}}
\newcommand{\lambdan}{\lambda_{n}}
\newcommand{\jnn}[1]{j_{\nu,#1}}
\newcommand{\jn}{\jnn{n}}

 \usepackage[pagewise,mathlines]{lineno}
\newcommand*\patchAmsMathEnvironmentForLineno[1]{%
	\expandafter\let\csname old#1\expandafter\endcsname\csname #1\endcsname
	\expandafter\let\csname oldend#1\expandafter\endcsname\csname end#1\endcsname
	\renewenvironment{#1}%
	{\linenomath\csname old#1\endcsname}%
	{\csname oldend#1\endcsname\endlinenomath}}%
\newcommand*\patchBothAmsMathEnvironmentsForLineno[1]{%
	\patchAmsMathEnvironmentForLineno{#1}%
	\patchAmsMathEnvironmentForLineno{#1*}}%
\AtBeginDocument{%
	\patchBothAmsMathEnvironmentsForLineno{equation}%
	\patchBothAmsMathEnvironmentsForLineno{align}%
	\patchBothAmsMathEnvironmentsForLineno{flalign}%
	\patchBothAmsMathEnvironmentsForLineno{alignat}%
	\patchBothAmsMathEnvironmentsForLineno{gather}%
	\patchBothAmsMathEnvironmentsForLineno{multline}%
}

\allowdisplaybreaks[1] 
\makeatletter

%% file: body-pages/introduction.tex
\section{Introduction}\label{intro}

Given $T>0$ and $\alpha \in (0,2)$, consider $ Q:=(0,T)\times (0,1)$
and suppose that $\omega\subset(0,1)$ is an open subset. We denote the characteristic function of $\omega$ by $\chi_\omega$. This work involves the distributed null controllability of 
\begin{equation}\label{dcp0}
\begin{cases}
y_{tt}- \displaystyle \left(x^\alpha y_{ x} \right)_x=g \chi_{\omega}, & (t,x)\in Q,\\
y(t,1)=0, & t\in (0,T),\\
\begin{cases}
y(t,0)=0, &\text{if } \alpha \in (0,1),\\
\text{or} & \\
\lim\limits_{x\to 0^+}(x^\alpha y_{x})(t,x)=0,&  \text{if } \alpha 
  \in [1,2),
\end{cases} &  t\in (0,T), \\
y(0,x)=y^0(x), y_{t} (0,x)=y^1(x),\ & x\in  (0,1),
\end{cases}
\end{equation}
and the boundary null controllability of 
\begin{equation}\label{bcp0}
\begin{cases}
z_{tt}- \displaystyle \left(x^\alpha z_x \right)_x=0, & (t,x)\in Q,\\
z(t,1)=h(t),& \text{ in } (0,T),\\
\begin{cases}
z(t,0)=0, & \text{if } \alpha \in (0,1),\\
\text{or} & \\
\lim\limits_{x\to 0^+}(x^\alpha z_x)(t,x)=0,& \text{if } \alpha \in [1,2),
\end{cases},&  t\in (0,T), \\
z(0,x)=z^0(x), z_t(0,x)=z^1(x), & x\in  (0,1),
\end{cases}
\end{equation}
in a sense that will be made precise below. The region $\omega$ is called the \textit{control domain}. 

Later, the initial data $(y^0,y^1)$ and $(z^0,z^1)$, as well as the controls $g$ and $h$, will be taken in some suitable functional spaces. We emphasize that the choice of these spaces is decisive to obtain the results we are dealing with. Since these spaces will only be presented in the next section, here we introduce the concept of null controllability without mentioning them.

We say that \eqref{dcp0} is \textit{null controllable} at time $T>0$ if, for any pair of initial data $(y^0,y^1)$, in a suitable Banach space, there exists a control \(g\in L^2(Q)\), acting on $\omega$, such that the solution $y$ of \eqref{dcp0} satisfies \begin{equation}\label{null}y(T,x)=y_t(T,x)=0, \ \ \forall x\in(0,1).    
\end{equation}
 
Likewise, \eqref{bcp0} is \textit{null controllable} at time $T>0$ if, for any pair of initial data $(z^0, z^1)$, in a suitable Banach space, there exists a control \(h\in L^2(0,T)\), acting on the boundary point $x=1$, such that the solution $z$ of \eqref{bcp0} satisfies
\begin{equation}\label{null2}
z(T,x)=z_t(T,x)=0 \ \ \forall x\in(0,1).    
\end{equation}
 
The null controllability of degenerate equations has attracted the attention of several mathematicians in the last two decades. In this period, the degenerate parabolic case has been the target of most publications (see \cite{cannarsa2006null,cannarsa2006regional,cannarsa2016global}, for instance). More recently, the null controllability of degenerate hyperbolic systems \eqref{dcp0} and \eqref{bcp0} has also been studied (see \cite{Zhang2018InteriorCO} and \cite{cannarsa2015wavecontrol}, respectively). 

In this paper, for a fixed initial data $(u^0, u^1)$, we will consider, for each $\varepsilon \in (0,1)$, the control domain 
\begin{equation}\label{domain}
\displaystyle \omega_{\varepsilon} :=(1-\varepsilon ,1) \subset (0,1).
\end{equation}
We intend to obtain a family of distributed state-control pairs $((u_{\varepsilon},v_{\varepsilon}))_{\varepsilon>0}$ solving \eqref{dcp0}, that is, 
\begin{equation}
	\begin{cases}
	u_{\varepsilon tt}- \displaystyle \left(x^\alpha u_{\varepsilon x} \right)_x=v_\ep \chi_{\omega_{\varepsilon}}, & (t,x)\in Q,\\
	u_{\varepsilon} (t,1)=0,& \text{ in } (0,T),\\
	\begin{cases}
	u_{\varepsilon} (t,0)=0, &{\text{if } \alpha \in (0,1),}\\
	\text{or}& \\
	\lim\limits_{x\to 0^+}(x^\alpha u_{\varepsilon x})(t,x)=0,&  \text{if } \alpha \in [1,2),
	\end{cases} &  t\in (0,T), \\
	u_{\varepsilon} (0,x)=u^0(x), u_{\varepsilon t} (0,x)=u^1(x),\ & x\in  (0,1),
\end{cases}
	\label{dcp}
	\end{equation}
and
 \begin{equation}\label{condTe}
      u_{\varepsilon} (T,x)=u_{\ep t}(T,x)=0,\  \forall x\in  (0,1),
 \end{equation}
with the following additional property: we can prove that $((u_{\ep},v_{\ep}))_{\ep >0}$ converges to $(u,h)$ in a suitable functional space, as $\ep \to 0$, with $(u,h)$ satisfying
  
\begin{equation}
	\begin{cases}
	u_{tt}- \displaystyle \left(x^\alpha u_x \right)_x=0, & (t,x)\in Q,\\
	u(t,1)=h(t),& \text{ in } (0,T),\\
	\begin{cases}
	u(t,0)=0, &\textcolor{black} {\text{if } \alpha \in (0,1),}\\
	\textcolor{black}{\text{or}}\\
	\textcolor{black}{\lim\limits_{x\to 0+ }(x^\alpha u_x)(t,x)=0,}& \textcolor{black}{\text{if } \alpha \in [1,2),}
	\end{cases},&  t\in (0,T), \\
	u(0,x)=u^0(x), u_t(0,x)=u^1(x), & x\in  (0,1)
	\end{cases}
	\label{bcp}
	\end{equation}	
and
 \begin{equation}\label{condT}
      u(T,x)=u_{t}(T,x)=0,\  \forall x\in  (0,1).
 \end{equation}

Roughly speaking, the desired family $((u_{\ep},v_{\ep}))_{\ep >0}$  will be obtained using the well-known  \textit{Hilbert Uniqueness Method} (HUM), where $v_\ep$ is a solution, in the sense of transposition, given by Lions-Magenes (see Definition \ref{trans}), of the homogeneous adjoint problem associated to \eqref{dcp}. Using rescaling $v_\ep=\frac{1}{\ep^3}\varphi_\ep$, we will be able to prove that $\varphi_\ep \stackrel{\ast}{\rightharpoonup} \varphi$ in $L^\infty(0,T;L^2(0,1)$. Consequently, we will achieve the desired convergence of $((u_{\ep},v_{\ep}))_{\ep >0}$ to $(u,h)$, where $v_\ep=\frac{1}{\ep^3}\varphi_\ep$
and $h=-\frac{1}{3}\varphi_x(t,1)$. The precise statement of this result is given in Theorem \ref{th4.2}.

It is worth mentioning that the boundary controllability of \eqref{bcp} has been studied in \cite{zhang2017null} for both the {\bf weakly degenerate} operator, when $\alpha\in(0,1)$ and  the {\bf strongly degenerate} operator, when $\alpha \in [1,2)$. On the other hand, for $\alpha \in (0,1)$, the distributed controllability of \eqref{dcp} has been established in \cite{Zhang2018InteriorCO}, for an arbitrary control domain $\omega\subset\subset(0,1)$. We emphasize that one of the contributions of this work is about the obtainment, for any $\alpha \in (0,2)$, of the null controllability of \eqref{dcp}, when the control domain is $\omega = \omega_{\varepsilon}$.

At this moment, let us briefly talk about some previous works which motivate the current one. In 1988, Zuazua used Lion's Hilbert uniqueness method in order to obtain the internal exact controllability for the wave equation, when the distributed control acts on an appropriate $\varepsilon$-neighborhood of some portion $\Gamma _{0}$ of the boundary (see \cite[Chapitre VII, section 2.3]{lions1988controlabilite} or \cite[Section III.2, Teorema 1]{zuazua1990controlabilidad}) 
  { or \hbox{\cite[Theorem 3.1]{zuazua2024exactcontrolwave}}}). Based on it, in \cite{fabre1992exact}, Fabre proved the exact boundary controllability of the wave equation as the limit of internal controllability, which means that, in the passage to the limit, when $\varepsilon \to 0$, the $\varepsilon$-neighborhood of $\Gamma _{0}$ shrinks to itself. In \cite{chaves2020boundary}, Chaves-Silva et al.~obtained a similar result for the heat equation. Recently, in \cite{araujo2022boundary}, we developed an analogous investigation for the degenerate heat equation case.

In this current work, we are focused on the corresponding investigation involving the degenerate wave equation case, that is, the boundary null controllability of the degenerate wave equation will be recovered as the limit of internal controllability. We are supposed to say that this goal has already been started in \cite{araujo2023regularity}, where we have obtained some crucial regularity results for the degenerate wave equations, in a neighborhood of the boundary.

It is important to note that our choice of the control domain \eqref{domain} is inspired by the aforementioned works of Zuazua and Fabre, where the control domain is an $\varepsilon$-neighborhood of a portion of the boundary. In our case, the boundary consists of exactly two points, $x = 0$ and $x = 1$. Thus, the general framework requires control domains $\omega_\varepsilon$ shrinking to either $x = 0$ or $x = 1$. 
However, the case where the domain shrinks to $x = 0$ is more delicate, because this is the point where the problem degenerates. This case remains open, with the main difficulty lying in obtaining similar regularity results to those proved in \cite{araujo2023regularity}, which rely on trace arguments for $x = 1$. 

As we mentioned earlier, the family $((u_\ep,v_\ep))_{\ep>0}$ that we seek is determined by HUM. The key of this approach pass by proving an \textit{observability inequality}
\begin{equation}\label{obs-ineq0}
         \n{v^0}{L^2(0,1)}^2+\n{v^1}{H_\alpha^{-1}}^2\leq C_{T,\alpha,\omega} \int_0^T\int_{\omega} |v|^2\dd,
\end{equation} 
satisfied for any solution of the \textit{adjoint problem} 
\begin{equation}\label{2.3}
	\begin{cases}
	v_{tt}- \displaystyle \left(x^\alpha v_{x} \right)_x=0, & (t,x)\in Q,\\
v(t,1)=0,& \text{ in } (0,T),\\
	\begin{cases}
	v(t,0)=0, &\textcolor{black}{\text{if } \alpha \in (0,1),}\\
	\textcolor{black}{\text{or}}& \\
	\textcolor{black}{\lim\limits_{x\to 0+ }(x^\alpha v_{x})(t,x)=0},&  \textcolor{black}{\text{if } \alpha \in [1,2),}
	\end{cases} &  t\in (0,T), \\
	v(0,x)=v^0(x),\ v_{t} (0,x)=v^1(x),\ & x\in  (0,1).
	\end{cases}
\end{equation}
The norm $\|\cdot\|_{H_\alpha^{-1}}$, mentioned in \eqref{obs-ineq0}, will be properly defined in the  next section.

 {The observability inequality {\eqref{obs-ineq0}} was proved 
in {\cite{Zhang2018InteriorCO}}, but only for the weakly degenerate case, where $\alpha\in(0,1)$. In our case, when $\omega=(1-\varepsilon,1)$, that is, the control domain depends on $\varepsilon$, that inequality takes the following form:}
\begin{equation}\label{obs-ineq1}
    \n{v^0}{L^2(0,1)}^2 + \n{v^1}{H_\alpha^{-1}} \leq C_{T,\alpha,\varepsilon} \int_0^T\int_{1-\varepsilon}^1 |v|^2 \dd.
\end{equation}
     {Although we can use it to obtain the null controllability of {\eqref{dcp}}, it is not sufficient to pass to the limit as $\varepsilon\to 0$, since:}
\begin{itemize}
    \item  {the dependence of $C_{T,\alpha,\varepsilon}$ on $\varepsilon$ is not explicitly given;}
    \item  {the minimal controllability time obtained by them depends on the domain, hence it also depends on $\varepsilon$.}
\end{itemize}
 {Besides that, it does not cover the strongly degenerate case.
We overcome these issues by proving a new observability inequality 
(Theorem {\ref{th-F-3.4}}) such that:}
\begin{itemize}
    \item  {$C_{T,\alpha,\varepsilon}$ is a constant of order $\varepsilon^{-3}$;}
    \item  {the minimal controllability time depends only on $\alpha$;}
    \item  {it holds for any $\alpha\in (0,2)$,}
\end{itemize}
 {provided the control domain touches the boundary $x=1$.}

 
  {
In order to justify  such dependence of $\varepsilon$ in the constant 
$C_{T,\alpha,\varepsilon}$, a result concerning the regularity of solutions by 
transposition is a crucial step. For the classical wave equation, as pointed out in 
Remark 3.1 of \mbox{\cite{fabre1992exact,fabre1993behavior}} and Remark 2.4 of  
\mbox{\cite{lions1988controlabilite}},  it is well known that if a solution by 
transposition has boundary normal derivative in {\(L^{2}(\Omega)\)}, then the 
corresponding initial data necessarily belong to the natural energy space, so that 
the solution is in fact weak. For degenerate wave equations, however, this property 
does not seem to be available in the literature. In Appendix \mbox{\ref{app-reg}}, 
we establish an analogous result by using a spectral approach based on 
Fourier–Bessel expansions combined with Ingham inequality. This method yields the 
desired regularity property, but introduces a new minimal controllability time 
{\(\tilde{T}_{\alpha}\)} when $\alpha \in (4/3,2)$. At this stage, it is not clear 
whether this additional restriction reflects an intrinsic feature of the degenerate 
wave equation or it is merely a limitation of the technique employed here.
}

The precise statements of our main results will be better understood after the presentation of some basics concepts. This work has the following structure: Section \ref{sec-pre} reunites definitions, results and general preliminaries, all related to the controllability problems involving the degenerate wave equation. Among them, the second section recalls us some recent regularity results in a neighborhood of the boundary, proved in \cite{araujo2023regularity};
Section \ref{sec-obs} is devoted to the homogeneous adjoint system \eqref{2.3}, for which some observability type inequalities will be gradually obtained.
The third section not only recollects some previously known results, but also provides some new ones originally developed here; 
Section \ref{main} only brings the statements of our main results 
{(Theorems} \ref{th-F-3.4}, \ref{th2.1-fabre} and \ref{th4.2}), 
considering that the notations and assumptions are all given in the previous 
sections.   {In a word,}   {Theorem} \ref{th-F-3.4}   {brings an 
observability inequality, with a positive constant} $C=C_{T,\alpha ,\ep}$ 
  {explicitly given in terms of} $\ep$,   {which is crucial to prove 
Theorem} \ref{th2.1-fabre},   {related to the exact internal 
controllability of the degenerate wave equation with the control domain 
being} $\omega_{\ep} = (1-\ep ,1)$.   {As a consequence of them, we can prove Theorem} \ref{th4.2},   {describing the boundary null controllability of the degenerate wave equation as the limit of internal controllability}; Section \ref{sec-null} is concerned with the proof of  Theorems \ref{th2.1-fabre} and \ref{th4.2}. In fact, the fifth section is an almost complete sketch of their proof, observing that there are some convergences which will just be rigorously justified in Section \ref{sec-conv}, where the whole passage to the limit strategy will be satisfactorily concluded; Finally, Section \ref{sec-prova1} is about the proof of Theorem \ref{th-F-3.4}, providing a crucial observability inequality, valid for each \textcolor{black}{$\alpha \in (0,2)$}, with an explicit constant described in terms of the parameter $\ep \in (0,1)$. Appendix \ref{app-well-pos} brings the proof of some well-podesness results, while Appendix \ref{app-trans} is about some properties of solutions by transpositions of degenerate wave equations.
   {We conclude this paper in Appendix \mbox{\ref{app-reg}}, presenting an additional regularity for solutions by transposition with normal derivative in \mbox{\(L^2(0,1)\)}. This final point is based on \mbox{\cite{gueye2014exact,HSSuTe2025}} and the references therein, relying on the approach given by the Fourier-Bessel series theory.}

%% file: body-pages/preliminaries.tex
\section{Preliminaries}\label{sec-pre}

The goal of this section is to establish the basic concepts concerning the well-posedness of a system like \eqref{dcp0}, as well as some previous regularity results that will have a key role in our study. Firstly, let us present some basic notation:

\begin{itemize}
    \item $\D(\Omega)$ is the space of test functions in $C^\infty(\Omega)$ with compact support in $\Omega$, see \cite{kesavan2015topics, bhattacharyya2012distributions}.
    \item $\D'(\Omega)$ is the space of distributions on $\Omega$, see \cite{kesavan2015topics, bhattacharyya2012distributions}.
    \item $\D(0,T;V)$ is the space of vector functions $v:(0,T)\longrightarrow V$ in $C^\infty(\Omega;V)$, compactly supported in $(0,T)$.
    \item $\D'(0,T;V)$ is the space of vector distributions $\varphi: \D(0,T)\longrightarrow V$, with the usual notion of convergence, see \cite[Chapter 9]{bhattacharyya2012distributions}.
    \item $\ac(\Omega)$ is the space of absolutely continuous functions on $\Omega$.
    \item $\acloc(\Omega)$ is the space of locally absolutely continuous functions on $\Omega$.
    \item $\langle \cdot,\cdot  \rangle_{V',V}$ denotes the standard duality pairing between the vector space $V$ and its dual $V'$.
    \item $(\cdot,\cdot)_H$ denotes the inner product defined in a Hilbert space $H$.
\end{itemize}



\subsection{Functional Spaces}

\begin{defn} [Weighted Sobolev spaces]
Consider $\alpha \in (0,1)$, for the \textbf{weakly degenerate case} (WDC), or $\alpha \in [1,2)$, for the \textbf{strongly degenerate case} (SDC).
\begin{itemize}
	\item[(I)] For the (WDC), we set 
\begin{equation*}
H_{\alpha}^1:= \{  u\in L^2(0,1);\ u \in \ac([0,1]),\
	x^{\alpha /2}u_x\in L^2(0,1) \mbox{ and } u(1)=u(0)=0\},
\end{equation*}
	equipped with the natural norm
\[ \|u\|_{H_{\alpha}^1}:=\left( \|u\|_{L^2(0,1)}^2+\| x^{\alpha /2} u_x\|_{L^2(0,1)}^2 \right) ^{1/2},\]
\item[(II)] For the (SDC),
\begin{equation*}
H_{\alpha}^1:= \{  u\in L^2(0,1);\ u\in \acloc((0,1]),\ x^{\alpha /2} u_x\in L^2(0,1) \mbox{ and } u(1)=0\},
\end{equation*}
and the norm keeps the same;

\item[(III)] In both situations, the (WDC) and the (SDC), 
and
\[
H_{\alpha}^2:= \{  u\in H_{\alpha}^1;\   {x^{\alpha} u_x}\in H^1(0,1) \}
\]
with the norm
$\|u\|_{H_{\alpha}^2}:=\left( \|u\|_{H_{\alpha}^1}^2+\|(  {x^{\alpha} u_x})_x\|_{L^2(0,1)}^2 \right) ^{1/2}$. 
\end{itemize}
\end{defn}	

{For properties of these spaces, we refer to {\cite{cannarsa2015wavecontrol}}. In particular, they proved in Proposition 2.5 (III)  that, for $\alpha\in [1,2)$, if $u\in \haa$, then}
\[  {\lim\limits_{x\to0^+}x^\alpha u_x(x)=0}.\]
{This is essential to guarantee that a classical solution satisfies the Neumann condition in the strongly degenerate case.}


Another important space in this context is $H_\alpha^{-1}=(H_\alpha^1)'$, the dual space of $H_\alpha^1$.

For each $u\in H_\alpha^1$, we define $-(x^\alpha u_x)_x\in H_\alpha^{-1}$ by 
\begin{equation}\label{def-iso}
\dpa{-(x^\alpha u_x)_x}{v}:=\int_0^1x^\alpha u_xv_x\,dx,\ \forall v\in H_\alpha^1.
\end{equation}
This establishes the bounded linear operator  $u\in \ha \longmapsto -(x^\alpha u_x)_x\in \had$. Actually, Lax-Milgram theorem yields that it is indeed an isomorphism. To see that, let us recall the following Poincar\'e inequality, proved in \cite[Proposition 2.2]{cannarsa2015wavecontrol}:
\begin{equation*}
\n{u}{\ha} \leq C \n{x^{\alpha/2}  {u_x}}{L^2(0,1)}, \forall u\in H^1_\alpha.  \end{equation*}
As a consequence, the norm $|u|_{1,\alpha}:=\n{x^{\alpha/2}u_x}{L^2(0,1)}$ is equivalent to $\n{u}{\ha}$. This provides a continuous,  coercive and symmetric bilinear form $a:\ha\times \ha \longrightarrow \R$, given by
\[a(u,v)=\into x^\alpha u_xv_x\, dx. \]
Therefore, given $\xi \in \had$, from Lax-Milgram theorem, there exists a unique $\hat{\xi} \in H^1_\alpha$ such that 
\begin{equation}\label{dual-prod}
\int_0^1x^\alpha \hat{\xi}_x w_x\,dx=\dpa{\xi}{w},\ \forall w\in H_\alpha^1.
\end{equation}
In other words, $u=\hat{\xi}\in \ha$ is the unique solution to the equation
\begin{equation}\label{ellip-eq}
-(x^\alpha u_x)_x=\xi \text{ in } \had.
\end{equation}
This gives us the aforementioned isomorphism. Let us denote its inverse by
\begin{equation}
\widehat{\ }: \xi\in \had \longmapsto \hat{\xi} \in \ha    
\end{equation}
Then, we can define the following inner product on $\had$:
\begin{equation}\label{inp-dual}
    \iphad{\xi}{\zeta}:=\into x^\alpha \hat{\xi}_x\hat{\zeta}_x\,dx,\ \forall \xi,\zeta \in \had.
\end{equation}

Note that the norm induced by the inner product \eqref{inp-dual} is exactly the usual one considered on $\had$.

   {
 Recall that any  $v\in L^2(0,1)\hookrightarrow\had$, using the identification $v\equiv\pl{v}{\cdot} \in \had$.
 }
   {
Thus, 
 combining \eqref{dual-prod} and \eqref{def-iso} we can see that} 
\[
 {\pl{v}{w}=\langle v,w\rangle_{H_\alpha^{-1},H_\alpha^1}=\into x^\alpha \hat{v}_x w_x\,dx=\langle-(x^\alpha \hat{v}_x)_x,w\rangle_{H_\alpha^{-1},H_\alpha^1},\ \forall w\in \ha.}
\]


We denote by \(\habd\), the dual of \(\haa\). For each \(u\in L^2(0,1)\), we define \((x^\alpha u_x)_x\in \habd\)  by
\begin{equation*}
    \dpaa{(x^\alpha u_x)_x}{w}=\into u (x^\alpha w_x)_x\,dx.
\end{equation*}
It is well defined, since
\begin{equation*}
\begin{split}
\left|\dpaa{(x^\alpha u_x)_x}{w}\right| & \leq \into |u| |(x^\alpha w_x)_x|\,dx\\
& \leq \nl{u}\nl{(x^\alpha w_x)_x}\\
& \leq \nl{u}\n{w}{\haa}, \ \forall w\in \haa.
\end{split}
\end{equation*}
In other words, we have the bounded linear operator 
\[ u\in L^2(0,1) \longmapsto  (x^\alpha u_x)_x\in \habd.\]

The following result is an important tool to study degenerate equations. Its proof can be found in \cite{cannarsa2008carleman,cannarsa2006null}. 

\begin{prop}[Hardy-Poincar\'e Inequality]\label{hardy}
    Assume that $\alpha\in (0,1)$. For any $u\in H_\alpha^1$, we have \[\int_0^1x^{\alpha-2}|u|^2\,dx\leq \frac{4}{(1-\alpha)^2}\int_0^1x^\alpha|u_x|^2\,dx.\]
\end{prop}

\subsection{Well-posedenss results}

Let us define the notion of solutions we are dealing with. We begin with the weak solution of the elliptic degenerate problem. 

\begin{defn}\label{def-weak-elip}
Given $\xi \in \had$, we say that $u\in \ha$ is a  \textbf{weak solution} of the degenerate elliptic problem 
    \begin{equation}\label{DEP}
    \begin{cases}
        -(x^\alpha u_x)_x=\xi, & \text{ in } (0,1),\\
        u(1)=0,\\
        \begin{cases}
        u(0)=0, & \alpha \in (0,1)\\
            \lim\limits_{x\to0^+}(x^\alpha u_x(x))=0, & \alpha \in [1,2),
        \end{cases}
    \end{cases}
\end{equation}
\end{defn}
if 
\begin{equation}\label{var-form}
    \into x^\alpha u_x w_x\,dx=\dpa{\xi}{w}, \forall w\in \ha.
\end{equation}

Recall that $u=\hat{\xi}$, given by  \eqref{dual-prod}, is the unique solution of problem \eqref{DEP}.

\begin{rem} 
Note that the elliptic problem \eqref{DEP} has Dirichlet boundary conditions for the (WDC). On the other hand, for the (SDC), \eqref{DEP} has a Neumann boundary condition at $x=0$ and a Dirichlet one at $x=1$. As pointed out in \cite[Section 3.2.2]{kesavan2015topics}, Dirichlet boundary conditions, also called \textbf{essential boundary conditions}, have to be imposed a \textit{priori} in the space $\ha$, while the Neumann boundary conditions, also called \textbf{natural boundary conditions}, do not. \textcolor{black}{Despite that, as we mention before, the Neumann condition can be recovered for a strong solution $u\in H_\alpha^2$. }  
%
%
\end{rem}

Next, let us specify which kind of solution to \eqref{pb1} we will deal with. Let us consider the problem
\begin{equation}
	\begin{cases}
	y_{tt}- \displaystyle \left(x^\alpha y_{ x} \right)_x=f, & (t,x)\in Q,\\
	y(t,1)=0,& \text{ in } (0,T),\\
	\begin{cases}
	y(t,0)=0, &\text{if } \alpha \in (0,1),\\
	\text{or}& \\
	  {\lim\limits_{x\to0^+}(x^\alpha y_x )(t,x)=0},&  \text{if } \alpha \in [1,2),
	\end{cases} &  t\in (0,T), \\
	y(0,x)=y^0(x), y_{t} (0,x)=y^1(x),\ & x\in  (0,1).
	\end{cases}
	\label{pb1}
	\end{equation}
\begin{defn}\label{weak} Given  $f\in L^1(0,T;L^2(0,1))$ and $(y^0,y^1)\in H^1_\alpha\times L^2(0,1)$, we say that
\[
\displaystyle y\in C^0([0,T];H^1_\alpha)\cap C^1([0,T];L^2(0,1)) 
\]
is a \textbf{weak solution} of \eqref{pb1}
if the following properties hold:
\begin{enumerate}[(a)]
\item
\begin{equation}\label{form-int}
   \frac{d}{dt}(y_t(t),w)_{L^2(0,1)}-\iphad{(x^\alpha y_x)_x(t)}{w}=(f(t),w)_{L^2(0,1)}
\end{equation}
in the sense of $\mathcal{D}'(0,T)$, for any $w\in \ha$, where $\iphad{\cdot}{\cdot}$ is defined in \eqref{inp-dual};
\item $y(0,\cdot)=y^0$ in $\ha$ and $y_t (0,\cdot)=y^1$ in $L^2(0,1)$.
\end{enumerate}
\end{defn}

Observe that \eqref{form-int} means that
\begin{equation}\label{form-int2}
\begin{split}
& \intq \Big(-y_t(t,x) w(x) \phi'(t) +x^\alpha y_x(t,x)w_x(x)\phi(t)\Big) \dd\\
& \qquad \qquad =\intq f(t,x)w(x)\phi(t)\dd, 
\end{split}
\end{equation}
for all $\phi\in \mathcal{D}(0,T)$ and $w\in \ha$.

\textcolor{black}{Concerning the existence of weak solutions for \eqref{pb1}, we state next a well-posedness result, whose proof will be given in Appendix \ref{app-well-pos}.}

\begin{prop}\label{well-pos}
	Given  $f\in L^1(0,T;L^2(0,1))$ and $(y^0,y^1)\in H^1_\alpha\times L^2(0,1)$, there exists a unique weak solution $y\in C^0([0,T];H^1_\alpha)\cap C^1([0,T];L^2(0,1))$ of \eqref{pb1}.	In addition, there exists a positive constant $C=C(T,\alpha)$ such that
\begin{multline}\label{ineq1}
\sup_{t\in[0,T]}\left( \n{y_t(t,\cdot )}{L^2(0,1)}^2+\n{y(t,\cdot )}{H_\alpha^1}^2 \right) \\
\leq C\left(\n{f}{L^1(0,T;L^2(0,1))}^2+\n{y^0}{H_\alpha^1}^2 +\n{y^1}{L^2(0,1)}^2\right).    
\end{multline}
\end{prop}

Associated with \eqref{pb1}, we have the following energy functional
\begin{equation*}
    E(t):=\frac{1}{2}\into \left(|y_t(t,x)|^2+x^\alpha |y_x(t,x)|^2\right) \,dx,
\end{equation*}
where $t\in (0,T)$. The previous result establishes that 
\[E(t)\leq C\left(\n{f}{L^1(0,T;L^2(0,1))}^2+E(0)\right),\]
an expected fact for hyperbolic equations. The next result is known as ``hidden regularity", another inherited property from the hyperbolic theory. Like the previous one, these results were also discussed in \cite{cannarsa2015wavecontrol,zhang2017null}.

\begin{prop}\label{prop-trace}
    For any weak solution $y$ of \eqref{pb1}, we have $y_x(\cdot,1)\in L^2(0,T)$ and 
    \begin{equation}
        \int_0^T|y_x(t,1)|^2\,dt\leq C\left(\|f\|_{L^1(0,T;L^2(0,1))}^2+E(0)\right).
    \end{equation}
\end{prop}

In the following,  we will present the definition of solutions for \eqref{pb1}, with less regular initial data, proposed by Lions-Magenes (see \cite[on page 47]{lions1988controlabilite}).


Let us also consider the following backward in time problem
\begin{equation}\label{back-pb}
	\begin{cases}
	z_{tt}- \displaystyle \left(x^\alpha z_{x} \right)_x=g, & (t,x)\in Q,\\
	z(t,1)=0,& t\in (0,T),\\
 \begin{cases}
     z(t,0)=0, & \text{ if } \alpha \in (0,1)\\
     \lim\limits_{x\to 0^+}(x^\alpha z_x)(t,x)=0& \text{ if } \alpha \in [1,2)
 \end{cases}& t\in (0,T),\\
	z(T,x)=z_{t} (T,x)=0, &  x\in (0,1).
 \end{cases}
	\end{equation}
The change of variable $t\longmapsto T-t$ transforms \eqref{back-pb} into \eqref{pb1} with zero as initial data    {and \(g\in L^1(0,T;L^2(0,1))\), which implies it is well-posed, by Proposition \ref{well-pos}.
}

\begin{defn}\label{trans}
 Given  $f\in L^1(0,T;L^2(0,1))$ and $(y^0,y^1)\in L^2(0,1)\times H^{-1}_\alpha $, we say that  $y\in L^\infty(0,T;L^2(0,1))$  is a \textbf{very weak solution}  
(or a \textbf{solution by transposition}) of \eqref{pb1}
if, for each $ F\in {\color{black} L^1(0,T;L^2(0,1)) }$, 
\begin{equation*}
        \intq yF\,dxdt=-\pl{y^0}{\theta_t(0,\cdot )}+\dpa{y^1}{\theta(0,\cdot )}+\intq f\theta \,dxdt,
\end{equation*}
where $\theta=\theta(t,x)$ is a weak solution of \eqref{back-pb} with $g=F$. 

\end{defn}

\begin{rem}
{It is important to observe that the initial conditions are not required in the definition of solution by transposition. Nevertheless, they can be recovered from the integral formulation, proceeding as in {\cite{medeiros2013introduction}}. In Appendix {\ref{app-trans}}, we will present a proof when $f=0$.}
\end{rem}

Analogously, we will give the definition of solution by transposition for the boundary controllability problem.

\begin{defn}\label{trans-b}
 Given $h\in L^2(0,T)$ and $(u^0,u^1)\in L^2(0,1)\times H^{-1}_\alpha $, we say that $u \in L^\infty(0,T;L^2(0,1))$  is a \textbf{very weak solution}  
(or a \textbf{solution by transposition}) of \eqref{bcp}
    if, for each $F\in L^1(0,T;L^2(0,1))$, 
\begin{equation*}
    \intq uF\,dxdt=-\pl{u^0}{\theta_t(0,\cdot)}+\dpa{u^1}{\theta(0,\cdot)}+\int_0^T h(t)\theta_x(t,1) \,{dt},
\end{equation*}
where $\theta=\theta(t,x)$ is a weak solution of \eqref{back-pb} with $g=F$. 
\end{defn}

The proof of the following well-posedness result will be given in Appendix \ref{app-well-pos}.
\begin{prop}\label{well-pos-trans}
	Given  $f\in L^1(0,T;L^2(0,1))$ and $(y^0,y^1)\in L^2(0,1)\times H^{-1}_\alpha$, there exists a unique solution by transposition  $y\in C^0([0,T];L^2(0,1))\cap C^1([0,T];H^{-1}_\alpha)$ of \eqref{pb1}.	In addition, there exists a positive constant $C=C(T,\alpha)$ such that
	\begin{equation}\label{ineq-2.3}
 \begin{split}
& 	\sup_{t\in[0,T]}\left( \n{y(t,\cdot )}{L^2(0,1)}^2+\n{y_t(t,\cdot )}{H_\alpha^{-1}}^2 \right) \\
& \qquad	\leq C\left(\n{f}{L^1(0,T;L^2(0,1))}^2+\n{y^0}{L^2(0,1)}^2+\n{y^1}{H_\alpha^{-1}}^2 \right).
 \end{split}
	\end{equation}
\end{prop}

\subsection{Further results}

Finally, we will present an observability inequality for solutions of \eqref{2.3}, proved in \cite[Theorem 3.3]{zhang2017null}. For any $\alpha\in (0,2)$, let us set 
\begin{equation}\label{Ta}
    T_\alpha=\frac{4}{2-\alpha}.
\end{equation}

\begin{prop}\label{ThObsZhang}
  Given $T>T_\alpha$, for any $(v^0,v^1)\in \ha\times L^2(0,1)$, there exists $C=C(T,\alpha)>0$, such that any solution $v$ of \eqref{2.3} satisfies
 \begin{equation}\label{ObsZhang}
\n{v^0}{\ha}^2+\n{v^1}{L^2(0,1)}^2\leq C \int_0^T |v_x(t,1)|^2\,dt.
 \end{equation}
\end{prop}

In order to become the reading easier, we recollect the two main theorems proved in \cite{araujo2023regularity}.
They play a key role in the proof of our main results, stated in Section \ref{main}.  

\begin{thm}\label{th3.2-fabre92}
Given $\ep_0 \in (0,1)$, there exists a positive constant  $C=C(T,\alpha,\ep_0)$ such that, for any $(u^0,u^1)\in H_\alpha^1\times L^2(0,1)$ and $f\in L^1(0,T;L^2(0,1))$, if $u$ is a weak solution to \eqref{pb1}, then
\begin{equation*}
    \frac{1}{\ep^3}\intw |u(t,x)|^2 \dd  \leq C\left(\n{f}{L^1(0,T;L^2(0,1))}^2+\n{u^0}{H^1_\alpha}^2+\n{u^1}{L^2(0,1)}^2\right),\ \forall \ep \in (0,\ep _0 ].
\end{equation*}
\end{thm}

In order to state the second one, we need to introduce a definition, {mentioned several times throughout this paper.}

\begin{defn}\label{epdef}
     A  family of functions $(\varphi^0_\ep,\varphi^1_\ep,h_\ep )_{\ep\in (0,1)}\in  L^2(\Omega)\times H^{-1}_\alpha\times L^1(0,T;L^2(\Omega))$ is said to be an \textit{$\ep-$family} associated with $(\varphi^0,\varphi^1,h)\in L^2(\Omega)\times H^{-1}_\alpha\times L^1(0,T:L^2(\Omega))$ if
\begin{equation*}
\begin{aligned}
   & h_\ep \rightharpoonup h && \text{ in } L^1(0,T:L^2(\Omega)),\\
    & \varphi_\ep^0 \rightharpoonup \varphi^0 &&  \text{ in } L^2(\Omega),\\
    & \varphi_\ep^1 \rightharpoonup \varphi^1 && \text{ in } H^{-1}_\alpha, \text{ as } \ep\to 0^+
\end{aligned}
\end{equation*}
Given an $\ep-$family for $(\varphi^0,\varphi^1,h)$, let $\varphi_\ep$ be  the solution by transposition of \eqref{pb1}, with $(y^0,y^1,f)=(\varphi_\ep^0,\varphi_\ep^1 ,h_\ep)$. It will be called an $\ep-$\textit{solution}.
From Proposition \ref{well-pos-trans},
\begin{equation*}
    \varphi_\ep\in C^0([0,T];L^2(0,1))\cap C^1([0,T];H^{-1}_\alpha)
\end{equation*}
and
\begin{equation}\label{phi-limit}
\varphi_\ep \stackrel{\ast}{\rightharpoonup} \varphi  \text{ in } L^\infty(0,T;L^2(0,1)), 
\end{equation}
 where $\varphi$ is the solution by transposition of \eqref{pb1}, with $(y^0,y^1,f)=(\varphi^0,\varphi^1,h)$. We will say that 
{$\varphi$ is the  \textit{limit} of the $\ep-$solutions $(\varphi_{\ep} )_{\ep >0}$.}
When $h_\ep=0$, we will simply denote the $\ep-$family $(\varphi^0_\ep,\varphi^1_\ep,0)$  by $(\varphi^0_\ep,\varphi^1_\ep)\in  L^2(\Omega)\times H^{-1}_\alpha$.
\end{defn}

\begin{thm}\label{th3.1}
{Under the notations given in Definition {\ref{epdef}}}, let $(\varphi^0_\varepsilon,\varphi^1_\ep,h_\varepsilon,)$ be an $\ep-$family associated to $(\varphi^0,\varphi^1,h)\in L^2(\Omega)\times H^{-1}_\alpha\times L^1(0,T:L^2(\Omega))$ and denote by $\varphi_\ep$ its $\ep-$solution. Let  $\varphi$ be the limit of $(\varphi_{\ep} )_{\ep >0}$. If 
\begin{equation}\label{3.4}
    \frac{1}{\ep^3}\intw |\varphi_\ep(t,x)|^2 \dd \leq C,
\end{equation}
where $C$ does not depend on $\ep$, then $\varphi_x(\cdot,1)\in L^2(0,T)$ and 
\begin{equation}\label{liminf}
    \frac{1}{3}\n{\varphi_x(\cdot,1)}{L^2(0,T)}^2
    \leq \liminf_{\ep\to 0^+}\left(\frac{1}{\ep^3}\intw |\varphi_\ep (t,x)|^2\dd \right).
\end{equation}
\end{thm}

%% file: body-pages/sec3.tex
\section{Homogeneous adjoint system}\label{sec-obs}

In this section, we will present several equivalent norms involving the solutions of the homogeneous system \eqref{2.3}. In fact, these results are known as \textit{observability inequalities} and have an important role in the framework of controllability. 

Firstly, let us refer to the energy conservation principle which has been presented in \cite{cannarsa2015wavecontrol}:

\begin{lem}\label{energy}
    For any $(v^0,v^1)\in H^1_\alpha\times L^2(0,1)$ and $v$ solution of \eqref{2.3}, we have \[E(t)=E(0), \ \forall t\in[0,T].\]
\end{lem}

Combining Propositions \ref{prop-trace} and \ref{ThObsZhang} we have the following:

\begin{prop}\label{norm.equiv1}
    For any $T>T_\alpha$, there exist two constants $A,B>0$, only depending on $T$ and $\alpha$, such that,  for any $(v^0,v^1)\in H_\alpha^1\times L^2(0,1)$ and $v$ the solution of \eqref{2.3}, we have 
\[\|v^0\|^2_{H_\alpha^1}+\|v^1\|^2_{L^2(0,1)}\leq A\int_0^T|v_x(t,1)|^2\,dt \leq B(\|v^0\|^2_{H_\alpha^1}+\|v^1\|^2_{L^2(0,1)}).\]
\end{prop}

In what follows, some results involving other equivalent norms will be given.

\begin{prop}\label{norm.equiv2}
   For any $T>T_\alpha$, there exist two  constants $A,B>0$, only depending on $T$ and $\alpha$, such that,  for any  $(v^0,v^1)\in H_\alpha^1\times L^2(0,1)$ and $v$ the solution of \eqref{2.3}, we have 
\begin{equation}\label{ineq.norm.equiv2}
\|v^0\|^2_{H_\alpha^1}+\|v^1\|^2_{L^2(0,1)}\leq A\intq |v_t|^2\,dx\,dt \leq B(\|v^0\|^2_{H_\alpha^1}+\|v^1\|^2_{L^2(0,1)}).
\end{equation}

\end{prop}

\begin{proof}

The second inequality is an immediate consequence of the energy estimate presented in Proposition \ref{well-pos}. In this case, let us focus on the first one.

{Firstly, let us prove it for a more regular initial data. Namely, if $(v^0,v^1)\in \haa\times \ha$, we know that $v\in C^0([0,T];\haa)\cap C^1([0,T];\ha)$ is a strong solution of {\eqref{2.3}.}} Thus, let us set $\rho(t)=t^2(T-t)^2$ and $\eta(x,t)=\rho(t)v(x,t)$. Multiplying the equation $v_{tt}-(x^\alpha v_x)_x=0$ by $\eta$ and integrating in $Q$, we get
 \begin{equation}\label{n1}  
\intq \rho|v_t|^2\,dx\,dt+\intq \rho_tvv_t\,dx\,dt=\intq \rho x^\alpha |v_x|^2\,dx\,dt.
\end{equation} 

On the other hand, since $\rho_t^2\leq C\rho$, using Young's inequality with $\delta>0$,  we get 
\begin{equation*}
\begin{split}
\left|\intq \rho_t vv_t\,dx\,dt\right| & \leq  C \delta\intq \rho |v|^2\,dx\,dt+\frac{1}{4\delta}\intq |v_t|^2\,dx\,dt\\
& \leq  C \delta\intq \rho x^{\alpha_1-2}|v|^2\,dx\,dt+\frac{1}{4\delta}\intq |v_t|^2\,dx\,dt
\end{split}
\end{equation*}

where $\alpha_1\in(0,2)-\{1\}$ and $\alpha<\alpha_1$. Moreover, from Hardy-Poincar\'e inequality, given in Proposition \ref{hardy}, we have \begin{equation*}\int_0^T\int_0^1\rho  x^{\alpha_1-2}|v|^2\,dx\,dt\leq C\int_0^T\int_0^1\rho x^{\alpha_1}|v_x|^2\,dx\,dt\leq C\int_0^T\int_0^1\rho x^\alpha|v_x|^2\,dx\,dt.
\end{equation*} 
Hence,
\begin{equation}\label{n2}
    \left|\intq \rho_t vv_t\,dx\,dt\right|\leq  C\delta\intq \rho x^{{\alpha}}|v_x|^2\,dx\,dt+\frac{1}{4\delta}\intq |v_t|^2\,dx\,dt.
\end{equation} 
Using \eqref{n2} in \eqref{n1}, if $\delta$ is sufficiently {small, we deduce} that \[\intq \rho x^\alpha|v_x|^2\,dx\,dt\leq C\intq |v_t|^2\,dx\,dt.\]
Note that $\int_0^T\rho(t)\,dt=T^5/30$. Thus, from Lemma \ref{energy}, 
\begin{multline}
\frac{T^5}{30}(\|v^0\|^2_{H_\alpha^1}+\|v^1\|^2_{L^2(0,1)})=\int_0^T\rho(t)\,dt(\|v^0\|^2_{H_\alpha^1}+\|v^1\|^2_{L^2(0,1)})=\int_0^T\rho(t)\,dt\, 2E(0)\\=\int_0^T 2\rho(t) E(t)\,dt=\intq \rho\left(|v_t|^2+x^\alpha|v_x|^2\right)\,dx\,dt \leq C \intq |v_t|^2\,dx\,dt.
\end{multline}

Finally, if $(v^0,v^1)\in \ha\times L^2(0,1)$, we can take a sequence $(v_n^0,v_n^1)_{n\in \mathbb{N}}\in \haa\times \ha$ such that 
\[v_n^0\to v^0 \text{ in } \ha \text{ and }\ v_n^1\to v^1 \text{ in } L^2(0,1),\]
From Proposition \ref{well-pos}, we have 
\[v_n \to v \text{ in } C^0([0,T];\ha) \text{ and } v_{nt}\to v_t \text{ 
 in } C^0([0,T];L^2(0,1)), \]
 where $v_n$ and $v$ are the solutions of \eqref{2.3} with initial data $(v_n^0,v_n^1)$ and $(v^0,v^1)$, respectively. Therefore, since $v_n$ satisfies the inequality \eqref{ineq.norm.equiv2}, passing to the limit as $n\to +\infty$, we have the desired result.
\end{proof}

\begin{rem}
In the rest of the paper, whenever we perform integration by parts for a weak solution, we have in mind the same previously given approach: we firstly obtain the estimates for a strong solution and then apply a density argument.
\end{rem}

\begin{prop}\label{norm.equiv3}
For any $T>T_\alpha>0$, there exist two constants $A,B>0$, only depending on $T$ and $\alpha$, such that, for any  $(v^0,v^1)\in L^2(0,1)\times H_\alpha^{-1}$, the very weak solution $v$ of \eqref{2.3} satisfies 
\[\|v^0\|^2_{L^2(0,1)}+\|v^1\|^2_{H_\alpha^{-1}}\leq A\intq |v|^2\,dx\,dt \leq B\left(\|v^0\|^2_{L^2(0,1)}+\|v^1\|^2_{H_\alpha^{-1}}\right).\]
\end{prop}

\begin{proof}
    Again, the second inequality comes immediately from the energy estimate presented in Proposition \ref{well-pos-trans}.

    In order to obtain the first one, let us take $(v^0,v^1)\in L^2(0,1)\times H_\alpha^{-1}$ and let $\varphi\in H_\alpha^1$ be the solution to the elliptic problem \eqref{ellip-eq}, with $\xi=v^1$. Setting 
    \[
w(t,x)=\int_0^tv(s,x)\,ds+\varphi(x),
    \]
    we have that $w$ is the weak solution of \eqref{2.3}, with the initial data $(\varphi,v^0)\in H_\alpha^1\times L^2(0,1)$ 
 {(see Proposition {\ref{lifting}})}. Hence, we can use Proposition \ref{norm.equiv2} to deduce that there exist positive constants $A,B>0$ such that 
 \[\|\varphi\|^2_{H_\alpha^1}+\|v^0\|^2_{L^2(0,1)}\leq A\intq |w_t|^2\,dx\,dt \leq B\left(\|\varphi\|^2_{H_\alpha^1}+\|v^0\|^2_{L^2(0,1)}\right).
\]
   Finally, since $w_t=v$ and $\|\varphi\|_{H_\alpha^1} \geq \|v^1\|_{H_\alpha^{-1}}$, the desired result follows.
\end{proof}

\begin{prop}\label{norm.equiv4}
 Given  $T>T_\alpha$ and $\ep_0 \in[0,1)$, there exists  a constant $C= C(T,\alpha,\ep_0) >0$ such that,  for any $(v^0,v^1)\in H_\alpha^1\times L^2(0,1)$, $v$ the solution of \eqref{2.3}, and  {for any}  $\ep \in (0,\varepsilon_0 )$, we have 
\[\|v^0\|^2_{H_\alpha^1}+\|v^1\|^2_{L^2(0,1)}\leq \frac{C}{\varepsilon}\int_0^T\int_{1-\varepsilon}^1(|v_t|^2+x^\alpha|v_x|^2)\,dx\,dt.\]
\end{prop}
\begin{proof}
    Let us take a cut-off function $h\in C^1([0,1])$ satisfying $0\leq h\leq1$ in $[0,1]$, $h=0$ in $[0,1-\varepsilon]$, $h=1$ in $[1-\frac{\varepsilon}{2},1]$ and 
\begin{equation}\label{b0}|h_x|\leq \frac{C}{\varepsilon}\ \text{ in } \ \left[1-\varepsilon,1-\frac{\varepsilon}{2}\right],
\end{equation}
where the constant $C>0$ does not depend on $\varepsilon$.

Now let us define $\sigma(t,x)=t(T-t)h(x)$. Multiplying the equation $v_{tt}-(x^\alpha v_x)_x=0$ by $\sigma x^\alpha v_x$ and integrating in $Q$, we obtain
\begin{equation}\label{n3}
\intq v_{tt}\sigma x^\alpha v_x\,dx\,dt=\intq \sigma x^\alpha v_x(x^\alpha v_x)_x\,dx\,dt.
\end{equation}
Integrating by parts, we see that 
\begin{equation*}
\begin{split}
\intq v_{tt}\sigma x^\alpha v_x\,dx\,dt& = -\int_0^T\int_0^1\sigma_t x^\alpha v_x v_t\,dx\,dt-\int_0^T\int_0^1\sigma x^\alpha v_{xt} v_t\,dx\,dt\\
& = -\int_0^T\int_0^1\sigma_t x^\alpha v_x v_t\,dx\,dt+\frac{1}{2} \int_0^T\int_0^1 (x^\alpha \sigma)_x |v_t|^2\,dx\,dt\\
& = -\int_0^T\int_0^1\sigma_t x^\alpha v_x v_t\,dx\,dt+\frac{1}{2} \int_0^T\int_{1-\ep}^1 \alpha x^{\alpha-1} \sigma |v_t|^2\,dx\,dt\\ & \hspace{0.4cm}+\frac{1}{2} \int_0^T\int_{0}^1  x^{\alpha} \sigma_x |v_t|^2\,dx\,dt.
\end{split}
\end{equation*}
and 
\begin{equation*}
\begin{split}
\intq \sigma x^\alpha v_x(x^\alpha v_x)_x\,dx\,dt & = 
\frac{1}{2}\intq\sigma (x^{2\alpha}|v_x|^2)_x\,dx\,dt\\
& =-\frac{1}{2}\intq x^{2\alpha}\sigma_x|v_x|^2\,dx\,dt+\frac{1}{2}\int_0^T\sigma(t,1)|v_x(t,1)|^2\,dt    
\end{split}
\end{equation*}
Using these identities in \eqref{n3}, we take
\begin{equation*}
\begin{split}
\frac{1}{2}\int_0^T\sigma(t,1)|v_x(t,1)|^2\,dt= & -\int_0^T\int_0^1\sigma_t x^\alpha v_x v_t\,dx\,dt+\frac{1}{2} \int_0^T\int_{1-\ep}^1 \alpha x^{\alpha-1} \sigma |v_t|^2\,dx\,dt\\
& +\frac{1}{2} \int_0^T\int_{0}^1  x^{\alpha} \sigma_x |v_t|^2\,dx\,dt+\frac{1}{2}\intq x^{2\alpha}\sigma_x|v_x|^2\,dx\,dt.
\end{split}
\end{equation*}
Now, let us estimate each integral on the right-hand side. For the first one, we have
\begin{equation*}
\begin{split}
\int_0^T\int_0^1|\sigma_t x^\alpha v_x v_t|\,dx\,dt & =\intw x^\alpha |\sigma_t v_xv_t|\dd \\
& \leq C \intw |v_t||x^{\alpha/2}v_x|\dd\\
& \leq C \intw \left(|v_t|^2 + x^{\alpha}|v_x|^2\right)\dd\\
& \leq \frac{C}{\ep} \intw \left(|v_t|^2 + x^{\alpha}|v_x|^2\right)\dd.
\end{split}
\end{equation*}
For the second one, we must pay attention to the term $x^{\alpha-1}$. For $\alpha\in [1,2)$ it is bounded by 1, but, for $\alpha\in (0,1)$, it is bounded by $(1-\ep_0)^{\alpha-1}$. Hence,
\begin{equation*}
\begin{split}
\frac{1}{2} \int_0^T\int_{1-\ep}^1 \alpha x^{\alpha-1} \sigma |v_t|^2\,dx\,dt 
&\leq C_{T,\alpha,\ep_0} \int_0^T\int_{1-\ep}^1 |v_t|^2\,dx\,dt\\
& \leq \frac{C}{\ep} \int_0^T\int_{1-\ep}^1 |v_t|^2\,dx\,dt.
\end{split}
\end{equation*}
For the last {two integrals}, it suffices to recall that $|h_x|\leq C/\ep$ in $[1-\ep,1-\ep/2]$ and $h_x$ vanishes in $[0,1]\setminus [1-\varepsilon, 1-\varepsilon /2]$, following that
\begin{equation*}
\frac{1}{2} \int_0^T\int_{0}^1  x^{\alpha} \sigma_x |v_t|^2\,dx\,dt\leq \frac{C}{\ep}\intw |v_t|^2\dd.
\end{equation*}
and
\begin{equation*}
\frac{1}{2}\intq x^{2\alpha}\sigma_x|v_x|^2\,dx\,dt\leq \frac{C}{\ep}\intw x^{\alpha}|v_x|^2\,dx\,dt.
\end{equation*}
As a consequence,
\begin{equation*}
\int_0^T\sigma(t,1)|v_x(t,1)|^2\,dt\leq \frac{C}{\ep} \intw \left(|v_t|^2 + x^{\alpha}|v_x|^2\right)\dd.
\end{equation*}

Notice that, since $\sigma(t,1)$ is not bounded from below, we can not apply the observability inequality \eqref{ObsZhang} directly.

So that, our next step is to derive an observability inequality within the interval $[\delta, T-\delta]$, where $\sigma$ is bounded from below.

To do that, since $T>T_\alpha$,  we can take $\delta= (T-T_\alpha)/4$, which implies $T-2\delta>T_\alpha$. Define $w(s,x)=v(s+\delta,x)$, with $0\leq s\leq T-2\delta$ and $0\leq x\leq 1$. We can see that $w$ is a solution of \eqref{2.3} with the initial data given by $w^0=v(\delta,x)$ and $w^1=v_t(\delta,x)$. Hence, Proposition \ref{ThObsZhang} yields a constant $C=C(T,\alpha)>0$ such that
\begin{equation*}
\n{w^0}{\ha}^2+\n{w^1}{L^2(0,1)}^2\leq C \int_0^{T-2\delta}w_x^2(s,1)\,ds.
 \end{equation*}
 Since $w_x(s,x)=v_x(s+\delta,x)$, we change variables to get
\begin{equation*}
\n{v(\delta,\cdot)}{\ha}^2+\n{v_t(\delta,\cdot)}{L^2(0,1)}^2\leq C\int_\delta^{T-\delta}v_x^2(t,1)\,dt.
 \end{equation*}
Finally, from the energy conservation, given by Lemma \ref{energy}, we have $E(0)=E(\delta)$, which leads us to
\begin{equation*}
\begin{split}
\|v^0\|^2_{H_\alpha^1}+\|v^1\|^2_{L^2(0,1)}& \leq C\int_\delta^{T-\delta}|v_x(t,1)|^2\,dt\leq C\int_\delta^{T-\delta}\sigma(t,1)|v_x(t,1)|^2\,dt\\
& \leq  \frac{C}{\varepsilon}\int_0^T\int_{1-\varepsilon}^1(|v_t|^2+x^\alpha|v_x|^2)\,dx\,dt.
\end{split}
\end{equation*}
\end{proof}

\begin{prop}\label{norm.equiv5}Given  $T>T_\alpha$ and $\ep_0 \in(0,1)$, there exists  a constant $C=C(T,\alpha,\ep_0)>0$ such that, for any $(v^0,v^1)\in H_\alpha^1\times L^2(0,1)$, $v$ solution of \eqref{2.3},  and  {for any} {$\varepsilon \in (0,\ep_0 )$}, we have \[\|v^0\|^2_{H_\alpha^1}+\|v^1\|^2_{L^2(0,1)}\leq \frac{C}{\varepsilon^3}\int_0^T\int_{1-\varepsilon}^1(|v_t|^2+|v|^2)\,dx\,dt.\] 
\end{prop}

\begin{proof}
    Let us take a cut-off function $h\in C^1([0,1])$ satisfying $0\leq h\leq1$ in $[0,1]$, $h=0$ in $[0,1-\varepsilon]$,  $h=1$ in $[1-\frac{\varepsilon}{2},1]$ and \begin{equation}\label{n4}
        \left|\frac{h_x^2}{h}\right|\leq \frac{C}{\varepsilon^2} \ \ \mbox{in} \ \ [1-\varepsilon,1],
    \end{equation}
    where the constant $C>0$ does not depend on $\varepsilon$. An explicit formula for this function can be found in \cite{lions1988controlabilite} (see the proof of Lemma 2.4 in Chapter 7).

    Now let us define $\sigma(t,x)=t(T-t)h(x)$. Multiplying the equation $v_{tt}-(x^\alpha v_x)_x=0$ by $\sigma v$ and integrating in $Q$, we obtain 
    \[
    \intq\sigma v(x^\alpha v_x)_x\,dx\, dt=\intq \sigma v v_{tt}\,dx\,dt.
    \]
    Integrating by parts, we get
    \begin{equation}\label{v1}
    \intq\sigma x^\alpha|v_x|^2\,dx\,dt=\intq(\sigma_tvv_t+\sigma|v_t|^2-\sigma_xvx^\alpha v_x)\,dx\,dt.
    \end{equation}
    In the following, we will estimate the terms on the right side of \eqref{v1}. For the first and the second ones, observe that
    \begin{equation*}
    \begin{split}
       \intq \sigma_tvv_t\,dx\,dt & =\int_0^T\int_{1-\varepsilon}^1\sigma_t v v_t\,dx\,dt\leq C\intw(|v|^2+|v_t|^2)\,dx\,dt\\
       & \leq \frac{C}{\varepsilon^2}\int_{1-\varepsilon}^1\int_0^T(|v|^2+|v_t|^2)\,dx\,dt
    \end{split}
    \end{equation*}
    and
    \begin{equation*}
        \intq \sigma|v_t|^2\,dx\,dt\leq \frac{C}{\varepsilon^2}\intw(|v|^2+|v_t|^2)\,dx\,dt.
    \end{equation*}
For the third one, we apply Young's inequality with $\delta>0$ and \eqref{n4} to obtain
\begin{equation*}
\begin{split}
\intq \sigma_xvx^\alpha v_x\,dx\,dt &\leq \delta\intq\sigma x^\alpha|v_x|^2\,dx\,dt+\frac{C}{\delta}\intq \left|\frac{\sigma_x^2}{\sigma}\right||v|^2\,dx\,dt\\
&\leq \delta\intq\sigma x^\alpha|v_x|^2\,dx\,dt+\frac{C}{\delta}\frac{1}{\varepsilon^2}\intw|v|^2\,dx\,dt.
\end{split}
\end{equation*}
Returning to \eqref{v1}, combining these last three estimates and taking $\delta$ sufficiently small, we deduce that \begin{equation}\label{v2}\intq \sigma x^\alpha|v_x|^2\,dx\,dt\leq \frac{C}{\varepsilon^2}\intw(|v_t|^2+|v|^2)\,dx\,dt.\end{equation}
Arguing as in the proof of Proposition \ref{norm.equiv4}, we can take $\delta>0$ such that $T-2\delta>T_\alpha$. So, the statement of Proposition \ref{norm.equiv4} can used with $w(s,x)=v(s+\delta,x)$, where $0\leq s\leq T-2\delta$, and $\ep$ replaced by $\ep/2$. In other words, there exists a  constant $C>0$, independent on $\varepsilon>0$, such that
\begin{equation}\label{v3}\|v^0\|^2_{H_\alpha^1}+\|v^1\|^2_{L^2(0,1)}\leq \frac{C}{\varepsilon}
\int_\delta^{T-\delta}\int_{1- \frac{\varepsilon}{2}}^1(|v_t|^2+x^\alpha|v_x|^2)\,dx\,dt.
\end{equation}

The result follows from \eqref{v3}, \eqref{v2} and the estimate 
\[
\int_\delta^{T-\delta}\int_{1-\frac{\varepsilon}{2}}^1x^\alpha|v_x|^2\,dx\,dt\leq C\int_\delta^{T-\delta}\int_{1-\frac{\varepsilon}{2}}^1\sigma x^\alpha|v_x|^2\,dx\,dt\leq  C\intq\sigma x^\alpha|v_x|^2\,dx\,dt.
\]  
\end{proof}

%% file: body-pages/sec4.tex
\section{Statements of the main results}\label{main}

At this moment, we are ready to state our main results. The first one is the observability inequality \eqref{obs-ineq1} with the dependence of the constant $C_{T,\alpha,\varepsilon}$ explicit on $\varepsilon$, where $\alpha\in(0,2)$. Then, as a consequence of this inequality, we  present a  null control result for \eqref{dcp} with $\omega=(1-\varepsilon,1)$. 

\begin{thm}\label{th-F-3.4} Let $\tilde{T}_\alpha$ defined in \eqref{Taa}. There exists $\varepsilon_0>0$ with the following property: for any $T>\tilde{T}_\alpha$ there exists a constant $C=C(T,\alpha)>0$ such that given $(v^0,v^1)\in L^2(0,1)\times H^{-1}_\alpha$ a solution $v$ of \eqref{2.3} satisfies
    \begin{equation}\label{obs-ineq}
         \n{v^0}{L^2(0,1)}^2+\n{v^1}{H_\alpha^{-1}}^{{2}}\leq \frac{C}{\ep^3} \intw |v|^2\dd, \ \ \ \forall\varepsilon\in(0,\varepsilon_0). 
    \end{equation}
\end{thm}

As a consequence of this observability inequality we can prove our next result,  the
exact internal controllability of the degenerate wave equation with the control
domain being $\omega=(1-\ep,1)$. This kind of result was originally proved by Zuazua
for the $n$-dimensional wave equation with the control domain being a neighborhood of
the boundary (see \cite[Chapitre VII, section 2.3]{lions1988controlabilite}, or
\cite[Section III.2, Teorema 1]{zuazua1990controlabilidad}, { or \hbox{\cite[Theorem 3.1]{zuazua2024exactcontrolwave}}})). As we have already  explained in introduction, the exact internal controllability of the degenerate wave equation, with $\omega \subset\subset \Omega$, was proved in {\cite{Zhang2018InteriorCO}}, but just for the weak degenerate case, that is, $\alpha \in (0,1)$. Our result covers the case $\alpha\in (0,2)$, with $\omega_\varepsilon=(1-\varepsilon,1)$, however the question remains open for a general control domain $\omega\subset\subset(0,1)$ when \(\alpha \in [1,2)\). 

 {The main tool that allowed us to prove the Theorem {\ref{th-F-3.4}} was the Theorem  {\ref{th3.4}}, which will be established in Section {\ref{sec-conv}}. This result describes the behavior of a solution in a neighborhood of the boundary point $x=1$. Therefore, this approach allowed us to  get the observability just for a control domain of the kind $\omega_\varepsilon=(1-\varepsilon,1)$. Fortunately, this inequality is enough to lead us to our results.}

\begin{thm}\label{th2.1-fabre}
   Given $T>\tilde{T}_\alpha$ and $\ep \in (0,1)$, for any $(u^0,u^1)\in \ha\times L^2(0,1)$, there exists $v_\ep\in L^\infty(0,T;L^2(0,1))$, solution of  \eqref{2.3} with initial data $(v_\ep^0,v_\ep^1)\in L^2(0,1)\times H^{-1}_\alpha$,   such that the corresponding weak solution $u_\ep$ of \eqref{dcp} satisfies \eqref{condTe}.  Moreover, the identity
    \begin{equation}\label{eq4.1}
     -(v_\ep^0,u^1)+\langle v_\ep^1,u^0\rangle =\intw v_\ep^2(t,x) \dd
    \end{equation}
holds and there exists a constant  $C=C(T,\alpha)>0$ such that 
    \begin{equation}\label{ineqTh4.1}
    \begin{cases}
       \displaystyle \n{v_\ep^0}{L^2(0,1)}+\ \nhd{v^1_\ep} \leq \frac{C}{\ep^3}\left(\nh{u^0}^2 + \nl{u^1}^2 \right)^{1/2};\\[1em]
        \displaystyle\intw v_\ep^2\dd \leq \frac{C}{\ep^3}\left(\nh{u^0}^2 + \nl{u^1}^2 \right).
    \end{cases}
    \end{equation}

\end{thm}

In the next result, we establish the convergence of a family solving \eqref{dcp} and \eqref{condTe} (distributed null controllability) to a solution of \eqref{bcp} and \eqref{condT} (boundary null controllability).

\begin{thm}\label{th4.2}
Given $T>\tilde{T}_\alpha$ and $\ep>0$, let us put $Q_\varepsilon=(0,T)\times\omega_\varepsilon$. For any $(u^0,u^1)\in \ha\times L^2(0,1)$, there exist $\varphi_\varepsilon\in L^2(\domw)$ and $u_\ep \in C([0,T];H^1_\alpha)\cap C^1([0,T];L^2(0,1))$, such that:
\begin{enumerate}[(a)]
    \item $u_\ep$ solves \eqref{dcp}, in the sense of Definition \ref{weak}, with $v_\ep:=\frac{1}{\ep^3}\varphi_\ep$, and satisfies \eqref{condTe};
    \item $u_\ep\rightharpoonup u$ and $\varphi_\ep \stackrel{\ast}{\rightharpoonup}\varphi$  in $L^\infty(0,T;L^2(0,1))$, as $\ep \to 0$. Moreover, $u$ solves \eqref{bcp}, in the sense of Definition \ref{trans-b}, with $h(t)=-\frac{1}{3}\varphi_x(t,1)\in L^2(0,T)$, and satisfies \eqref{condT}.
\end{enumerate}
\end{thm}
 

%% file: body-pages/sec5.tex
\section{Proof of Theorems \ref{th2.1-fabre} and \ref{th4.2}}\label{sec-null}

Firstly, let us establish the null controllability for \eqref{dcp} using HUM. 

It provides a family of distributed controls for which we will develop our convergence analysis. So that, it is a crucial step in this paper.

\begin{proof}[proof of Theorem \ref{th2.1-fabre}] 
    Given $(v^0,v^1)\in L^2(0,1)\times H^{-1}_\alpha$, let $v$ be the  solution by transposition of \eqref{2.3}, in the sense of Definition \ref{trans}. So, let $u_\ep$ be the weak solution to the backward in time problem \eqref{back-pb}, with $g=v\chi_{\omega_\ep}$. Therefore, Propositions \ref{well-pos} and \ref{well-pos-trans} allow us to define the continuous linear operator
    \[\Lambda_\ep: (v^0,v^1)\in L^2\times H^{-1}_\alpha \longmapsto (-u_{\ep t}(0,\cdot ),u_\ep(0, \cdot ))\in  L^2\times \ha.\]
It is sufficient to prove that $\Lambda_\ep$ is invertible. Indeed, once it is proved,  for $(u^0,u^1)\in \ha\times L^2(0,1)$, we can take $(v_\ep^0,v^1_{\ep}):=\Lambda_{\ep}^{-1}(-u^1,u^0)\in L^2(0,1)\times H^{-1}_\alpha$. So that, for this initial data, let $v_\ep$ be the solution \textcolor{black}{by transposition of} \eqref{2.3}  and $u_\ep$ be the \textcolor{black}{weak} solution to the backward in time problem  \eqref{back-pb}, with $g=v_\ep\chi_{\omega_\ep}$. Hence, we will have 
$(-u^1,u^0)=\Lambda_\ep(v_\ep^0,v^1_\ep)=(-u_{\ep t}(0,\cdot ),u_\ep(0, \cdot ))$, following that $(u_\ep,v_\ep)$ will be \textcolor{black}{a} solution to the control problem described in \eqref{dcp} and \eqref{condTe}.

Let us prove that $\Lambda_\ep$ is really invertible, by applying \textit{Lax-Milgram Theorem}. Setting $F=L^2\times H^{-1}_\alpha$, observe that $L^2\times \ha\hookrightarrow F'$, using the duality product
\[\left\langle (\varphi^0,\varphi^1), (\psi^0,\psi^1) \right\rangle_{F',F}:= (\psi^0,\varphi^0)+\langle \psi^1,\varphi^1\rangle,\]
for each $(\varphi^0, \varphi^1)\in L^2 \times \ha$.

Let $a: F\times F \longrightarrow \R$ be the continuous bilinear form induced by $\Lambda_\ep$, given by
\[a\left((v^0,v^1),(\psi^0,\psi^1)\right):=\left\langle \Lambda_\ep(v^0,v^1), (\psi^0,\psi^1) \right\rangle_{F',F}=-(\psi^0,u_{\ep t}(0,\cdot ))+\langle \psi^1, u_\ep(0,\cdot )\rangle.\]
We only need to check that $a$ is coercive. In fact, since $u_\ep$ solves \eqref{back-pb}, we multiply $u_{\ep_{tt}}- \displaystyle \left(x^\alpha u_{\ep_{x} }\right)_x=v\chi_{\omega_\ep}$
by $v$ and integrate by parts to yield

\begin{equation}\label{eq4.1-2}
-(v^0,u_{\ep t}(0,\cdot ))+\langle v^1,u_{\ep }(0,\cdot )\rangle =\intw v^2(t,x) \dd.    
\end{equation}
In this case, \eqref{eq4.1-2} and the observability inequality \eqref{obs-ineq}
give us the coercivity of $a$, as required.

Therefore, taking $(-u^1,u^0)\in L^2\times \ha\hookrightarrow F'$, \textit{Lax-Milgram Theorem} assures the existence of a unique $(v^0,v^1)\in F=L^2\times H^{-1}_\alpha$ such that
\begin{equation*}
\begin{split}
\left\langle \Lambda_\ep(v^0,v^1), (\psi^0,\psi^1) \right\rangle_{F',F}& = a\left((v^0,v^1),(\psi^0,\psi^1)\right)\\
& =\left\langle (-u^1,u^0), (\psi^0,\psi^1) \right\rangle_{F',F},\ \forall (\psi^0,\psi^1)\in F,    
\end{split}
\end{equation*}
following that $ \Lambda_\ep(v^0,v^1)=(-u^1,u^0)$. Furthermore, we can see that the identity \eqref{eq4.1} comes immediately from \eqref{eq4.1-2}.

Next, it remains to prove those inequalities stated in \eqref{ineqTh4.1}. Indeed, observe that using the observability Inequality \eqref{obs-ineq} and the identity \eqref{eq4.1}, we obtain
\begin{equation*}
    \n{v_\ep^0}{L^2(0,1)}^2 + \nhd{v^1_\ep}^2 \leq \frac{C}{\ep^3}  \left(\n{v_\ep^0}{L^2(0,1)}^2 + \nhd{v^1_\ep}^2 \right)^{1/2}\left(\nh{u^0}^2 + \nl{u^1}^2 \right)^{1/2},
\end{equation*}
whence we get
\begin{equation*}
   \left(\n{v_\ep^0}{L^2(0,1)}^2 + \nhd{v^1_\ep}^2 \right)^{1/2} \leq \frac{C}{\ep^3}  \left(\nh{u^0}^2 + \nl{u^1}^2 \right)^{1/2},
\end{equation*}
which provides the first inequality in \eqref{ineqTh4.1}. For the second one, we use the  identity \eqref{eq4.1} again and the previous inequality, as follows:
\begin{equation*}
    \begin{split}
\intw v_\ep^2\dd & \leq 2\left(\n{v_\ep^0}{L^2(0,1)}^2 + \nhd{v^1_\ep}^2 \right)^{1/2}\left(\nh{u^0}^2 + \nl{u^1}^2 \right)^{1/2}\\
& \leq  \frac{C}{\ep^3}  \left(\nh{u^0}^2 + \nl{u^1}^2 \right).
 \end{split}
\end{equation*}
It completes the proof.
\end{proof}

Now we are ready to prove Theorem \ref{th4.2}. It is concerned with the convergence of the family  $((u_\ep,v_\ep))_{\ep \in (0,1)}$ obtained in Theorem \ref{th2.1-fabre}, as $\ep\to 0^+$. For reasons that will be more clear later, we will consider this convergence in the sense of transposition. To be more precise, we desire to prove that $(u_\ep,v_\ep)$ converges to a solution $(u,h)$ of the boundary control problem \eqref{bcp}, in the sense of transposition.

\begin{proof}[Proof of Theorem \ref{th4.2}] Let us recall that $(u,h)$ is a solution by transposition of \eqref{bcp} if, given \textcolor{black}{$F\in L^1 (0,T;L^2 (0,1))$, we have}
\begin{equation*}
    \intq uF\,dxdt=-(u^0,\theta_t(0,\cdot ))+\langle u^1,\theta(0,\cdot )\rangle+\int_0^T h(t)\theta_x(t,1) \,dt,
\end{equation*}
where $\theta$ solves \eqref{back-pb} with $g=F$. On the other hand, since $u_\ep\in C([0,T];\ha)$ is a solution by transposition of \eqref{dcp}, the relation
\begin{equation}\label{ue}
    \intq u_\ep F\,dxdt=-(u^0,\theta_t(0,\cdot ))+\langle u^1,\theta(0,\cdot )\rangle+\intw v_\ep\theta \,dxdt.
\end{equation}
holds. Therefore, we intend to prove that 
\begin{equation}\label{first}
\intq u_\ep F\,dxdt \to \intq uF\,dxdt
\end{equation}
and
\begin{equation}\label{second}
\intw v_\ep\theta \,dxdt\to \int_0^T h(t)\theta_x(t,1) \,dt,
\end{equation}
as $\ep\to 0^+$.
Note that the convergence \eqref{first} holds because $u_\ep \stackrel{\ast}{\rightharpoonup} u$  in $L^\infty(0,T;L^2(0,1))$, up to a subsequence, which can be achieved  by  proving that $(u_\ep )_{\ep \in (0,1)}$ is uniformly bounded in $(L^1(0,T;L^2(0,1)))'$. In fact, since $u_\ep$ is a weak solution of \eqref{dcp}, we can write 
\[\langle\langle u_\ep, F \rangle\rangle=-(u^0,\theta_t(0,\cdot))+\langle u^1,\theta(0,\cdot)\rangle+\intw v_\ep\theta \,dxdt,\]
where $\langle\langle \cdot, \cdot \rangle\rangle$ denotes the duality pair involving  $L^\infty(0,T;L^2(0,1))$ and $L^1(0,T;L^2(0,1))$. 
Thus, 
\begin{equation*}
\left|\langle\langle u_\ep, F \rangle\rangle\right|\leq \nl{u^0}\nl{\theta_t(0,\cdot )}+\nhd{u^1}\nh{\theta(0,\cdot )}+\n{v_\ep}{L^2(Q_\ep)}\n{\theta}{L^2(Q_\ep)}.
\end{equation*}
Since $\theta$ solves \eqref{back-pb}, using the energy estimate given in \eqref{ineq1} and {Theorem} \ref{th3.2-fabre92}, we have 
\[\max\{ \n{\theta(0,\cdot )}{\ha},\n{\theta_t(0,\cdot )}{L^2(0,1)} \}\leq C\n{F}{L^1(0,T;L^2(0,1))} \]
and
\[\n{\theta}{L^2(Q_\ep)}\leq C\ep^{3/2}\n{F}{L^1(0,T;L^2(0,1))}. \]
These two last inequalities, together with \eqref{ineqTh4.1} give us 
\begin{equation*}
\left|\langle\langle u_\ep, F \rangle\rangle\right|\leq C\n{F}{L^1(0,T;L^2(0,1))},
\end{equation*}
where $C$ does not depend on $\ep$, as required. 

Next, we will focus on the convergence  \eqref{second}, which is more delicate, because a family of integrals over $(0,T)\times (1-\ep,1)$ is supposed to converge to an integral over $(0,T)$, and inequality \eqref{ineqTh4.1} does not assure it immediately. To overcome this, we need to characterize $v_\ep$ by a rescaling,  in order to obtain a uniform boundedness.

Observe that the function $\varphi_\ep=\ep^3v_\ep$  is the solution to the homogeneous problem \eqref{2.3}, with $\varphi^0_\ep:=\ep^3 v_\ep^0$ and $\varphi^1_\ep:=\ep^3 v^1_\ep$ as initial data. Hence, \eqref{ineqTh4.1} yields the following  uniform boundedness with respect to $\ep$:
\begin{equation}\label{ineq.unif}
\n{\varphi_\ep^0}{L^2(0,1)}+\n{\varphi_\ep^1}{H^{-1}_\alpha}\leq C\ \text{ and }\ \frac{1}{\ep^3}\intw |\varphi_\ep|^2 \dd \leq C.  
\end{equation}
It means that $((\varphi^0_\ep,\varphi^1_\ep))_{\ep \in (0,1)}$ is a bounded family in $L^2(0,1)\times H^{-1}_\alpha$. Hence, up to a subsequence, there exists $(\varphi^0,\varphi^1)\in L^2(0,1)\times \had$ such that 
\begin{equation*}
\begin{aligned}
    & \varphi_\ep^0 \rightharpoonup \varphi^0 &&  \text{ in } L^2(0,1);\\
    & \varphi_\ep^1 \rightharpoonup \varphi^1 && \text{ in } H^{-1}_\alpha.
\end{aligned}
\end{equation*}
Moreover, since $\varphi_\ep$ is the solution by transposition of \eqref{2.3}, we conclude that $\varphi_\ep$ is uniformly bounded in $L^\infty(0,T;L^2(0,1))$, which gives us
\[\varphi_\ep \stackrel{\ast}{\rightharpoonup} \varphi\  \text{ in } L^\infty(0,T;L^2(0,1)),\]
where $\varphi$ is a solution of \eqref{2.3} with initial data $(\varphi^0,\varphi^1)\in L^2(0,1)\times \had$.

{\color{black} It means that  $((\varphi^0_\ep,\varphi^1_\ep))_{\ep \in (0,1)}$ is an  $\ep-$family associated to $(\varphi^0,\varphi^1)$ and    $\varphi_\ep$ is its $\ep-$solution with limit $\varphi$. Furthermore, $\varphi_\ep$ satisfies \eqref{ineq.unif}}. Therefore, Theorem \ref{th3.1} guarantee that $\varphi_x(t,1) \in L^2(0,T)$. As a consequence, applying Proposition \ref{norm.equiv1}, we conclude that $(\varphi^0,\varphi^1)\in \ha \times L^2(0,1)$. Therefore, $\varphi$ is a solution of \eqref{2.3} with finite energy.

It remains to prove that
\[
\frac{1}{\ep^3}\intw \varphi_\ep \theta  \dd \to \frac{1}{3}\int_0^T \varphi_x(t,1) \theta_x(t,1)\, dt, \text{ as } \ep \to 0^+ ,
\]
for any $\theta = \theta (t,x)$ solving  \eqref{back-pb} with \textcolor{black}{$g=F\in L^1 (0,T;L^2(0,1))$}. Of course, it implies the convergence \eqref{second},  with $v_\ep=\frac{1}{\ep^3}\varphi_\ep$ and $h(t)=\frac{1}{3}\varphi_x(t,1)$.

To prove it, our strategy relies on setting functionals  
\[
G_\varepsilon ,G:H_\alpha^1\times L^2(0,1)\times L^1(0,T;L^2(0,1))\longrightarrow \mathbb{R}
\]
given by
\begin{equation}\label{G}
\begin{split}
G_\varepsilon(y^0,y^1,f):=\frac{1}{\ep^3}\int_0^T\int_{1-\varepsilon}^1\varphi_\varepsilon y\,dx\,dt,  \\
G(y^0,y^1,f):=\frac{1}{3}\int_0^T \varphi_x(t,1) y_x(t,1)\,dt,
\end{split}   
\end{equation}
 where $y$ is the solution of \eqref{pb1} and $(\varphi_\varepsilon)_{\ep \in (0,1)}$ is the family of controls previously defined. Our goal consists in proving the convergence 
 \[G_\ep \stackrel{\ast}{\rightharpoonup}  G, \text{ in }  H_\alpha^{-1}\times L^2(0,1)\times L^\infty(0,T;L^2(0,1)).\]
 This will be developed in the next section. 
 \end{proof}

%% file: body-pages/sec6.tex
\section{Passage to the limit}\label{sec-conv}

In this section, we will study the convergence of the family $(G_\ep)_{\varepsilon \in (0,1)}$, given in \eqref{G}. In order to do it, for each $\varepsilon \in (0,1)$, let us consider 
    $L_\varepsilon:L^2 (0,T; \haa )\longrightarrow \mathbb{R}$, given by \[L_\varepsilon v=\frac{1}{\varepsilon^3}\int_0^T\int_{1-\varepsilon}^1\varphi_\varepsilon v\,dx\,dt,\] where $\varphi_\varepsilon$ is the solution of \eqref{2.3} with $\varphi_\varepsilon^0$ and $\varphi_\epsilon^1$ as initial data. We will prove the following result.

\begin{thm}\label{th3.4} 
\textcolor{black}{Under the notations given in Definition \ref{epdef}}, {consider an $\ep-$family $((\varphi^0_\varepsilon,\varphi^1_\varepsilon))_{\ep >0}$ associated to   $(\varphi^0,\varphi^1)\in L^2(0,1)\times \had$, and denote by $\varphi_\ep$ its $\ep-$solution, with}
\[
\varphi_\ep \stackrel{\ast}{\rightharpoonup} \varphi  \text{ in } L^\infty(0,T;L^2(0,1)).
\]
Assume that there exist $C>0$ and $\varepsilon_0>0$, such that
\[\frac{1}{\varepsilon^3}\int_0^T\int_{1-\varepsilon}^1|\varphi_\varepsilon|^2\,dx\,dt\leq C, \ \ \forall\varepsilon\in (0,\varepsilon_0).
\]
Then:
\begin{itemize}
        \item[(a)] $\varphi_x(\cdot,1)\in L^2(0,T)$;
        \item[(b)] $G_\varepsilon\stackrel{\ast}{\rightharpoonup} G$ in $H_\alpha^{-1}\times L^2(0,1)\times L^\infty(0,T;L^2(0,1))$, where \[G(u_0,u_1,h)=\frac{1}{3}\int_0^T\varphi_x(t,1)u_x(t,1)\,dt,\] 
        and $u$ is the solution of \eqref{pb1}, with $(u_0,u_1,h) \in H_\alpha^{-1}\times L^2(0,1)\times L^\infty(0,T;L^2(0,1))$ as the initial data;
        \item[(c)] $\displaystyle\frac{1}{3}\int_0^T|\varphi_x(t,1)|^2\,dt\leq\liminf\frac{1}{\varepsilon^3}\int_0^T\int_{1-\varepsilon}^1|\varphi_\varepsilon|^2\,dx\,dt.$
 \end{itemize}
 \end{thm}

{We observe that (a) and (c) comes directly from Theorem {\ref{th3.1}}. In this case, let us {focus} on the convergence mentioned in (b). To be more precise, it will be a consequence of Corollaries {\ref{cor3.2}} and {\ref{cor3.4}}, given in the following. }

\begin{lem}\label{lema3.1},
     \textcolor{black}{Under the notations given in Definition \ref{epdef}}, let $((\varphi^0_\varepsilon,\varphi^1_\varepsilon))_{\ep >0}$ be an $\ep-$family associated with   $(\varphi^0,\varphi^1)\in L^2(0,1)\times \had$, denote by $\varphi_\ep$ its $\ep-$solution, and consider the limit $\varphi$ of $(\varphi_{\ep} )_{\ep >0}$. Assume that there exist a positive constant $C=C(T,\alpha)$ l{and} $\ep_0 >0$, such that \[\frac{1}{\varepsilon^3}\int_0^T\int_{1-\varepsilon}^1|\varphi_\varepsilon|^2\,dx\,dt\leq C, \forall \ep \in (0,\ep_0). \] 
     \textcolor{black}{Then, $(G_\ep )_{\ep >0}$ and $(L_\ep )_{\ep >0}$ are uniformly bounded in $(H_{\alpha}^{1} \times L^2 \times L^1 (0,T;L^2(0,1)))'$ and $(L^2 (0,T;H_{\alpha}^{2}))'$, respectively.}
\end{lem}

\begin{proof}
    To deduce that $G_\varepsilon$ is continuous, it is enough to use Hölder inequality and Theorem \ref{th3.2-fabre92}. Now, let us analyze the continuity of $L_\varepsilon$. Again, from Hölder inequality, we have
    \[
    \textcolor{black}{|L_\varepsilon v |^2} \leq C\frac{1}{\varepsilon^3}\int_0^T\int_{1-\varepsilon}^1|v|^2\,dx\,dt.
    \]
    As in the proof of Lemma 4.2 of \cite{araujo2022boundary}, we can see that \[|v(t,x)|^2\leq \left[(1-x)^2|v_x(t,1)|^2+\varepsilon(1-x)^2\int_{1-\varepsilon}^1|v_{xx}(t,r)|^2\,dr\right],\]
    whence
    \[ \textcolor{black}{|L_\varepsilon (v)|^2}\leq C(I_1+I_2),\]where \[I_1=\frac{1}{\varepsilon^3}\int_0^T\int_{1-\varepsilon}^1(1-x)^2|v_x(t,1)|^2\,dx\,dt\]
    and 
    \[I_2=\frac{1}{\varepsilon^3}\int_0^T\int_{1-\varepsilon}^1(1-x)^2\varepsilon\left(\int_{1-\varepsilon}^1|v_{xx}|^2\,dr\right)\,dx\,dt.\] 
 Applying Corollary 2.6 of \cite{araujo2022boundary}, we have 
    \[I_1 =\frac{1}{\varepsilon^3}\varepsilon^3\int_0^T|v_x(t,1)|^2\,dx\,dt\leq C_{\alpha} |v|^2_{L^2(0,T;H_\alpha^2)},\]
\textcolor{black}{where $C_\alpha >0$ only depends on $\alpha$}. On the other hand, for $\varepsilon\in (0,\frac{1}{2} )$, we can use the inequality (2.13) of \cite{araujo2022boundary} to deduce that 
    \begin{equation*}
        \begin{aligned}
        I_2&\leq& \frac{1}{\varepsilon^3}\int_0^T\int_{1-\varepsilon}^1\varepsilon(1-x)^2\left[\frac{4}{(1-\varepsilon)^{2\alpha}}+\frac{4\alpha^2}{(1-\varepsilon)^{2+\alpha}}\right]|v|^2_{H_\alpha^2}\,dx\,dt\\
        &=& \frac{\ep}{3}\left[\frac{4}{(1-\varepsilon)^{2\alpha}}+\frac{4\alpha^2}{(1-\varepsilon)^{2+\alpha}}\right]\int_0^T|v|^2_{H_\alpha^2}\,dt\ \leq \ \tilde{C}_{\alpha} |v|^2_{L^2(0,T;H_\alpha^2)},
        \end{aligned}
    \end{equation*}
    \textcolor{black}{where $\tilde{C}_{\alpha} >0$ only depends on $\alpha$.} 
    \end{proof}

\begin{cor}\label{cor3.2}
\textcolor{black}{Under the notations given in Definition \ref{epdef}}, let $((\varphi^0_\varepsilon,\varphi^1_\varepsilon))_{\ep >0}$ be an $\ep-$family associated with   $(\varphi^0,\varphi^1)\in L^2(0,1)\times \had$, denote by $\varphi_\ep$ its $\ep-$solution, and consider the limit $\varphi$ of $(\varphi_{\ep} )_{\ep >0}$. Assume that there exist a positive constant $C=C(T,\alpha)$ {and} $\ep_0 >0$, such that \[\frac{1}{\varepsilon^3}\int_0^T\int_{1-\varepsilon}^1|\varphi_\varepsilon|^2\,dx\,dt\leq C, \forall \ep \in (0,\ep_0 ).\] Then, there exist two linear and continuous functionals 
    \[{\tilde{G}} :H_\alpha^1 \times L^2(0,1) \times L^1(0,T;L^2(0,1)) \longrightarrow \mathbb{R} \text{ and } L:L^2(0,T; \haa )\longrightarrow \mathbb{R},\] 
    such that, up to subsequences, we have {$G_\varepsilon \stackrel{\ast}{\rightharpoonup} \tilde{G}$} in $[H_\alpha^1\times L^2(0,1)\times L^1(0,T;L^2(0,1))]'$ and \textcolor{black}{$L_\varepsilon  \stackrel{\ast}{\rightharpoonup}  L$} in $L^2(0,T; \haa )'$, as $\ep \to 0^+$.
\end{cor}
\begin{proof}
    {In fact, it is a consequence of having $(G_\ep )_{\ep >0}$ and $(L_\ep )_{\ep >0}$ uniformly bounded in $(H_{\alpha}^{1} \times L^2 \times L^1 (0,T;L^2(0,1)))'$ and $(L^2 (0,T;H_{\alpha}^{2}))'$, respectively.} 
\end{proof}

\begin{lem}\label{lema3.2}
    For any $v\in \mathcal{D}((0,T); \haa )$, we have
    \[L(v)=\frac{1}{3}\int_0^T\varphi_x(t,1)v_x(t,1)\,dt,\]
    where $L \in L^2(0,T; \haa )'$ is the functional mentioned in Corollary \ref{cor3.2}.
\end{lem}

\begin{proof}   {Let us fix $v\in \mathcal D((0,T); \haa )$. From the identity (2.15) of {\cite{araujo2022boundary}}, we know that}
    \[v(t,x)=-(1-x)v_x(t,1)+(1-x)V(t,x),\]
    where
    \[V(t,x)=\frac{1}{1-x}\int_x^1\int_s^1 \textcolor{black}{v_{xx}}(t,r)\,dr\,ds.\]
    Hence, \[L_\varepsilon(v)=A_\varepsilon+B_\varepsilon,
    \]
     where \[A_\varepsilon :=\frac{-1}{\varepsilon^3}\int_0^T\int_{1-\varepsilon}^1(1-x)\varphi_\varepsilon v_x(t,1)\,dx\,dt \ \ \ \mbox{and} \ \ \ B_\varepsilon :=\frac{1}{\varepsilon^3}\int_0^T\int_{1-\varepsilon}^1(1-x)\varphi_\varepsilon V(t,x)\,dx\,dt.\]
Assuming that $\ep \in (0,\frac{1}{2})$, we have $V\in L^2\left((0,T)\times (\frac{1}{2},1)\right)$. This leads us to 
    \begin{equation*}
        \begin{split}
            |B_\varepsilon|& \leq \left[\frac{1}{\varepsilon^3}\int_0^T\int_{1-\varepsilon}^1|\varphi_\varepsilon|^2\,dx\,dt\right]^{\frac{1}{2}}\left[\frac{1}{\varepsilon^3}\int_0^T\int_{1-\varepsilon}^1(1-x)\textcolor{black}{^{2}}|V|^2\,dx\,dt\right]^{\frac{1}{2}}\\
            & \leq C\left[\textcolor{black}{\frac{1}{\ep ^3}} \int_0^T\int_{1-\varepsilon}^1 \textcolor{black}{(1-x)^2}|V|^2\,dx\,dt\right]^{\frac{1}{2}} \\
            &\textcolor{black}{\leq C\left[ \frac{1}{3} \sup_{x\in [1-\ep ,1]} \n{V(\cdot ,x)}{L^2 (0,T)}^2 \right]^{\frac{1}{2}}.}
        \end{split}
    \end{equation*}
     \textcolor{black}{Applying \cite[Lemma 2.8]{araujo2022boundary}, we deduce that $|B_{\ep}| \to 0$, as $\varepsilon\to 0^+$}. At this point, it remains to deal with $A_\varepsilon$. To do that, let us take $\theta_\varepsilon\in H_\alpha^1$ such that \textcolor{black}{$(x^\alpha\theta_{\varepsilon x})_x=\varphi_\varepsilon^1$ in $H_\alpha^{-1}$} and define $$\phi_\varepsilon(t,x)=\int_0^t\varphi_\varepsilon(s,x)\,ds+\theta_\varepsilon(x).$$ 
 Observe that $\phi_\varepsilon(0,\cdot)=\theta_\varepsilon\in H_\alpha^1$, $\phi_{\varepsilon t}(0,\cdot)=\varphi^0_\varepsilon\in L^2(0,1)$ and \textcolor{black}{$\phi_{\varepsilon tt}-(x^\alpha\phi_{\varepsilon x})_x=0$.} This allows us to conclude that $\phi_\varepsilon \in C^0([0,T];H_\alpha^1)\cap C^1([0,T];L^2(0,1))$.

On the other hand, since $\varphi^1\in H_\alpha^{-1}$, there exists $\theta\in H_\alpha^1$ such that $$\langle \varphi^1,\xi\rangle_{H_\alpha^{-1}}=\int_0^1x^\alpha\theta_x\xi_x\,dx \ \ \forall\xi\in H_\alpha^1 .$$
So that, 
\[
    \int_0^1x^\alpha\theta_{\varepsilon x}\xi_x\,dx=\langle(x^\alpha \theta_{\varepsilon x})_x,\xi\rangle_{H_\alpha^{-1}}=\langle \varphi^1_\varepsilon,\xi\rangle_{H_\alpha^{-1}}\to\langle\varphi^1,\xi\rangle_{H_\alpha^{-1}}=\int_0^1x^\alpha\theta_x\xi_x\,dx ,\ \ 
\]
\textcolor{black}{as $\ep \to 0^+$, for each $\xi \in H_\alpha^1$.}
It means that
\begin{equation}\label{eq3.1}
        \theta_\varepsilon\rightharpoonup\theta \ \ \mbox{in} \ \ H_\alpha^1.
    \end{equation}
    Since $\varphi^0_\varepsilon\in L^2(0,1)$, we can also take $\psi^0_\varepsilon\in  \textcolor{black}{H_\alpha^2 }$ such that \textcolor{black}{$(x^\alpha \psi_{\varepsilon x}^{\textcolor{black}{0}})_x=\varphi^0_\varepsilon$ in $L^2(0,1)$} and define \[\psi_\varepsilon(t,x)=\int_0^t\phi_\varepsilon(s,x)\,ds+\psi^0_\varepsilon(x).\] 
    Thus, $\psi_\varepsilon(0,\cdot)=\psi^0_{\textcolor{black}{\ep}}\in  \haa $, $\psi_{\varepsilon t}(0,\cdot)=\theta_\varepsilon\in H_\alpha^1$ and \textcolor{black}{$\psi_{\varepsilon tt}-(x^\alpha\psi_{\varepsilon x})_x=0$.} Consequently, $\psi_\varepsilon \in C^0([0,T]; \haa )\cap C^1([0,T];H_\alpha^1)$. 
    
    Finally, we can rewrite $A_\varepsilon$ as follows:
\begin{equation*}
    \begin{split}
         A_\varepsilon&=\frac{-1}{\varepsilon^3}\int_0^T\int_{1-\varepsilon}^1(1-x)\phi_{\varepsilon t}(t,x)v_x(t,1)\,dx\,dt\\
         &=\frac{1}{\varepsilon^3}\int_0^T\int_{1-\varepsilon}^1(1-x)\phi_\varepsilon(t,x)v_{tx}(t,1)\,dx\,dt\\
         &=\frac{-1}{\varepsilon^3}\int_0^T\int_{1-\varepsilon}^1(1-x)\left[\int_x^1\phi_{\varepsilon x}(t,r)\,dr\right]v_{tx}(t,1)\,dx\,dt\\ 
         &=\frac{-1}{\varepsilon^3}\int_{1-\varepsilon}^1(1-x)\int_x^1\langle\phi_{\varepsilon x}(\cdot,r)-\phi_{\varepsilon x}(\cdot,1),v_{tx}(\cdot,1)\rangle_{L^2(0,T)}\,dr\,dx\\
         & \qquad \qquad-\frac{1}{3}\int_0^T\phi_{\varepsilon x}(t,1)v_{tx}(t,1)\,dt\\
         &=\frac{1}{\varepsilon^3}\int_{1-\varepsilon}^1(1-x)\int_x^1\langle\psi_{\varepsilon x}(\cdot,r)-\psi_{\varepsilon x}(\cdot,1),v_{ttx}(\cdot,1)\rangle_{L^2(0,T)}\,dr\,dx\\
         & \qquad\qquad \textcolor{black}{-\frac{1}{3}\int_0^T\phi_{\varepsilon x}(t,1)v_{tx}(t,1)\,dt }.
    \end{split}
\end{equation*}
     \textcolor{black}{Since 
     \[
     \psi_\varepsilon \in C^0([0,T]; \haa )\cap C^1([0,T];H_\alpha^1),
     \]
     we have
     \[
     \psi_{\ep x} \in L^2(0,T;H^1 (1-\ep,1)) \equiv H^1 ((1-\ep,1);L^2(0,T)) \hookrightarrow C^0 ([1-\ep,1];H^{-1} (0,T)).
     \]
     Arguing as in \cite[Lemma 2.8]{araujo2022boundary}, we obtain 
     }
{\color{black}
\begin{align*}
\displaystyle & \left|\frac{1}{\varepsilon^3}\int_{1-\varepsilon}^1(1-x)\int_x^1\langle\psi_{\varepsilon x}(\cdot,r)-\psi_{\varepsilon x}(\cdot,1),v_{ttx}(\cdot,1)\rangle_{L^2(0,T)}\,dr\,dx\right| \\
& \hspace{0.5cm}\leq \frac{C}{\varepsilon^3}\int_{1-\varepsilon}^1(1-x) \int_{x}^{1} \n{\psi_{\varepsilon x}(\cdot,r)-\psi_{\varepsilon x}(\cdot,1)}{H^{-1} (0,T)} drdx \\
& \hspace{0.5cm}\leq \frac{C}{\varepsilon^3} \cdot \frac{\ep^3}{3} \sup_{r\in [1-\ep ,1]} \n{\psi_{\varepsilon x}(\cdot,r)-\psi_{\varepsilon x}(\cdot,1)}{H^{-1} (0,T)},
\end{align*}
}
\textcolor{black}{where this last term converges to zero, as $\ep \to 0^+$.}

     \textcolor{black}{Finally, applying \cite[Corollary 2.6]{araujo2022boundary}, we know that
     \[
     \displaystyle \varphi_{\ep} \rightharpoonup\varphi \text{ weakly in } L^2 (0,T;H^2_\alpha ),
     \]
    which implies
     \[
\displaystyle  -\frac{1}{3}\int_0^T\phi_{\varepsilon x}(t,1)v_{tx}(t,1)\,dt = \frac{1}{3}\int_0^T \varphi_{\varepsilon x}(t,1)v_{x}(t,1)\,dt 
\to \frac{1}{3}\int_0^T \varphi_{x}(t,1) v_{x}(t,1)\,dt,
     \]
     as $\ep \to 0^+$. Summarizing, we have 
     \[
     \displaystyle L_\ep (v) = A_\ep +B_\ep \to \frac{1}{3}\int_0^T \varphi_{x}(t,1) v_{x}(t,1)\,dt,
     \]
     as $\ep \to 0^+$, for each $v\in \mathcal D((0,T); \haa )$, as desired.
     }
\end{proof}

Since $\mathcal D([0,T]; \haa )$ is dense in $L^2(0,T; \haa )$, from Lemmas \ref{lema3.1} and \ref{lema3.2}, we obtain the following result:

    \begin{cor}\label{cor3.4}
        For any $v\in L^2(0,T;\haa )$, we have \[L(v)=\frac{1}{3}\int_0^T\varphi_x(t,1)v_x(t,1)\,dt,\]
        where $L \in L^2(0,T;\haa)'$ is the functional mentioned in Corollary \ref{cor3.2}.
    \end{cor}

{\color{black} 
\begin{proof}[{Proof of Theorem} \ref{th3.4}(b)]
Take $(y^0,y^1,f)\in  H^1_{\alpha} \times L^2(0,1)\times L^1(0,T;L^2(0,1))$ and let $y$ be the corresponding solution of \eqref{pb1}. Consider sequences $(y^0_n)_{n=1}^{\infty}$, $(y^1_n)_{n=1}^{\infty}$ and  $(f_n)_{n=1}^{\infty}$ in $H^2_{\alpha}$, $H^1_{\alpha}$ and $L^1(0,T;H^1_{\alpha})$, respectively, such that
\[
\displaystyle y^0_n \to y^0 \text{  in  } H^1_{\alpha},
\]
\[
\displaystyle y^1_n \to y^1 \text{  in  } L^2 (0,1)
\]
and
\[
\displaystyle f_n \to f \text{  in  } L^1(0,T,L^2(0,1)).
\]
For each $n\in \mathbb N$, let $y_n$ be the solution of \eqref{pb1} having $(y^0_n, y^1_n ,f_n )$ as initial data. In particular, $y_n \in L^2(0,T;H^2_{\alpha})$. Since $G_\ep (y^0_n, v^1_n, f_n) = L_{\ep}(y_n)$, Corollaries \ref{cor3.2} and \ref{cor3.4} give us
\[
\displaystyle\lim_{\ep \to 0}  G_\ep (y^0_n, y^1_n, f_n) = L(y_n ) = G (y^0_n, y^1_n, f_n),
\]
\textcolor{black}{for each $n\in \mathbb N$, where $G$ is given in \eqref{G}}. We claim that $G_\ep (y^0, y^1, f) \to G(y^0, y^1, f) $, as $\ep \to 0$, which means that 
\[
G_\varepsilon\stackrel{\ast}{\rightharpoonup} G \text{ in } H_\alpha^{-1}\times L^2(0,1)\times L^\infty(0,T;L^2(0,1))
\]
\textcolor{black}{and $\tilde{G} = G$, where $\tilde{G}$ was obtained in Corollary \ref{cor3.2}}. Indeed, fix $r>0$ arbitrarily. So, there exists $m\in \mathbb N$ such that
\[
\n{(y^0_m, y^1_m, f_m) - (y^0,y^1,f)}{E}
< \min \left\{ \frac{r}{3C_0}, \frac{r}{3\|G\|_{E'}  }\right\},
\]
where $E=H^1_{\alpha} \times L^2(0,1) \times L^1(0,T;L^2(0,1))$ and $\displaystyle C_0 = \sup_{\ep \in (0,1)} \|G_\ep \|_{E'} >0$. In particular, there exists $\ep_1 \in (0,1)$ such that
\[
\displaystyle |G_\ep (y^0_m, y^1_m, f_m)-G (y^0_m, y^1_m, f_m)|<\frac{r}{3},
\]
for any $\ep \in (0,\ep_1 )$. Hence,
\begin{align*}
\displaystyle |G_\ep (y^0, y^1, f) - G (y^0, y^1, f) |
&\leq \n{G_\ep}{E'} \n{(y^0,y^1,f)-(y^0_m, y^1_m, f_m)}{E} \\
&\hspace{.5cm}+ |G_\ep (y^0_m, y^1_m, f_m)-G (y^0_m, y^1_m, f_m)| \\
&\hspace{.5cm}+\n{G}{E'} \n{(y^0_m, y^1_m, f_m) - (y^0,y^1,f)}{E} \\
&< C_0 \cdot \frac{r}{3C_0} + \frac{r}{3} + \|G\|_{E'} \cdot \frac{r} {3\|G\|_{E'}} \\
&=r,
\end{align*}
for any $\ep \in (0,\ep_1 )$. It concludes the proof.
\end{proof}
}

%% file: body-pages/sec7.tex
\section{Proof of Theorem \ref{th-F-3.4}}\label{sec-prova1}

This section is devoted to the obtainment of Theorem \ref{th-F-3.4}, which is a consequence of the following result.

\begin{thm}\label{th3.6}
    Let  $T_\alpha>0$ defined in \eqref{Ta}. There exist $C>0$ and $\varepsilon_0>0$ such that, for any $T>T_\alpha$, $(\phi^0,\phi^1)\in H_\alpha^1\times L^2(0,1)$, and $\phi$ solution of \eqref{2.3} with this data, we have \[\|\phi^0\|^2_{H_\alpha^1}+\|\phi^1\|^2_{L^2(0,1)}\leq C\left[\frac{1}{\varepsilon^3}\int_0^T\int_{1-\varepsilon}^1|\phi_t|^2\,dx\,dt\right],\ \forall \varepsilon\in(0,\varepsilon_0).\]
\end{thm}

Arguing as in the proof of Proposition \ref{norm.equiv3}, we can see that Theorem \ref{th3.6} leads to Theorem \ref{th-F-3.4}. So that, we will next focus on the proof Theorem \ref{th3.6}.

\begin{proof}
    The desired inequality will be proved by contradiction. Assuming that the result is false, for each $n\in \mathbb N$, there exist two sequences $(\varepsilon_n )_{n=1}^{\infty}$ in  {$(0,1)$}, converging to zero, and $(\bar{\phi}^0_n,\bar{\phi}^1_n)_{n=1}^{\infty}$ in $H_\alpha^1\times L^2(0,1)$ such that 
    \begin{equation}\label{eq3.3}        \|\bar{\phi}^0_n\|^2_{H_\alpha^1}+\|\bar{\phi}^1_n\|^2_{L^2(0,1)}>n\left[\frac{1}{\varepsilon_n^3}\int_0^T\int_{1-{\varepsilon_n}}^1|\bar{\phi}_{nt}|^2\,dx\,dt\right],
    \end{equation}
    for each $n\in \mathbb N$.
    This implies $\bar{\phi}_n\neq0$ and allows us to define

    \[\phi^0_n=\frac{\bar{\phi}^0_n}{\sqrt{\|\bar{\phi}^0_n\|^2_{H_\alpha^1}+\|\bar{\phi}^1_n\|^2_{L^2(0,1)}}}, \ \ \phi^1_n=\frac{\bar{\phi}^1_n}{\sqrt{\|\bar{\phi}^0_n\|^2_{H_\alpha^1}+\|\bar{\phi}^1_n\|^2_{L^2(0,1)}}}\]
    and
    \[\phi_n=\frac{\bar{\phi}_n}{\sqrt{\|\bar{\phi}^0_n\|^2_{H_\alpha^1}+\|\bar{\phi}^1_n\|^2_{L^2(0,1)}}}.\]

    Observe that $\phi_n$ is the solution of \eqref{2.3} with the initial data $(\phi^0_n , \phi^1_n ) \in H_\alpha^1\times L^2(0,1)$, with
    \begin{equation}\label{eq3.4}
        \|\phi^0_n\|^2_{H_\alpha^1}+\|\phi^1_n\|^2_{L^2(0,1)}=1
    \end{equation}and\begin{equation}\label{eq3.5}
        \frac{1}{\varepsilon_n^3}\int_0^T\int_{1-\varepsilon_n}^1|\phi_{nt}|^2\,dx\,dt<\frac{1}{n}.
    \end{equation}
    {At this point, we recall that Proposition} \ref{norm.equiv5} {assures that}

\begin{equation}\label{completo}
\|\phi^0_{n}\|^2_{H^1_\alpha}+\|\phi^1_{n}\|^2_{L^2(0,1)}\leq \frac{C}{\varepsilon_n^3}\int_0^T\int_{1-\varepsilon_n}^1 (|\phi_{nt}|^2+|\phi_{n}|^2 )\,dx\,dt.
\end{equation}
{Taking into account} \eqref{eq3.5} end \eqref{completo}, {if we prove that, up to subsequences, the convergence}

\begin{equation}\label{f1} \frac{1}{\varepsilon^3_{n}}\int_0^T\int_{1-\varepsilon_{n}}^1|\phi_{n}(t,x)|^2\,dx\,dt\to0, \emph{ as } n\to +\infty ,
\end{equation}
{holds, then we will achieve a conclusion that contradicts} \eqref{eq3.4}.

    {The remainder of this proof is devoted to check that, up to subsequences,} \eqref{f1} {can really be obtained. To do so, several steps will be necessary.} Firstly, from \eqref{eq3.4}, we deduce that there exists $(\phi^0,\phi^1)\in H_\alpha^1\times L^2(0,1)$ such that $\phi^0_n\rightharpoonup\phi^0$ in $H_\alpha^1$ and $\phi^1_n\rightharpoonup\phi^1$ in $L^2(0,1)$. Furthermore, 
\begin{equation}\label{eq3.5.0}
\phi_n \stackrel{\ast}{\rightharpoonup} \phi \ \ \text{in} \ \ L^\infty(0,T;H_\alpha^1) \ \ \mbox{and} \ \ \phi_{nt}\stackrel{\ast}{\rightharpoonup}\phi_t \ \ \text{in} \ \ L^\infty(0,T;L^2(0,1)),
    \end{equation}where $\phi$ is the solution of \eqref{2.3} with the data $\phi^0$ and $\phi^1$. 


Now, we can apply Theorem \ref{th3.4} to the sequence $(\phi_{nt})$ to deduce that $\phi_{tx}(\cdot,1)\in L^2(0,T)$ and, from \eqref{eq3.5}, we get
\[
\frac{1}{3}\int_0^T|\phi_{tx}(t,1)|^2\,dt\leq\liminf \frac{1}{\varepsilon_n^3}\int_0^T\int_{1-\varepsilon_n}^1|\phi_{nt}|^2\,dx\,dt=0.\] 
From Proposition \ref{B3}, $v:=\phi_t\in L^\infty(0,T;L^2(0,1))$ is a very weak solution for \eqref{2.3} with data $v^0=\phi_t(0,\cdot)$ and $v^1=\phi_{tt}(0,\cdot)$. On the other hand, $v_x(\cdot,1)=\phi_{tx}(\cdot,1)=0\in L^2(0,T)$ and we can apply Proposition \ref{reg-ND} to conclude that $v\in L^\infty(0,T;H_\alpha^1)$ is a weak solution for \eqref{2.3}. Hence, Proposition \ref{norm.equiv2} leads to $v^0=v^1=0$, i.e. $\phi_t(0,\cdot)=\phi_{tt}(0,\cdot)=0$. 
 

Furthermore, from \eqref{dual-prod}, we have
\[\dpa{\phi^0}{\xi}=-\int_0^1(x^\alpha \phi_x)_x(0,x)\xi(x)\,dx=-\int_0^1\phi_{tt}(0,x)\xi(x)\,dx=0, \ \forall\xi\in H_\alpha^1.\] and this implies that $\phi^0=0$. Therefore, $\phi=0$, since it is a solution of \eqref{2.3} with initial data  $\phi^0=\phi^1=0$.

Now, from the identity 
\[\phi_n(r,x)=\phi_n(t,x)+\int_t^r\phi_{nt}(s,x)\,ds,\]
we deduce
\begin{equation*}
\begin{split}
|\phi_n(r,x)|^2 & =|\phi_n(t,x)|^2+2\phi_n(t,x)\int_t^r\phi_{nt}(s,x)\,ds+\left(\int_t^r\phi_{nt}(s,x)\,ds\right)^2\\
& \geq |\phi_n(t,x)|^2+2\phi_n(t,x)\int_t^r\phi_{nt}(s,x)\,ds,
\end{split}
\end{equation*}
which leads us to 
\begin{multline*}
\frac{1}{\varepsilon^3_n}\int_0^T\int_{1-\varepsilon_n}^1|\phi_n(r,x)|^2\,dx\,dr\\ \geq \frac{T}{\varepsilon^3_n}\int_{1-\varepsilon_n}^1|\phi_n(t,x)|^2\,dx+\frac{2}{\varepsilon^3_n}\int_0^T\int_{1-\varepsilon_n}^1 \phi_n(t,x)\left(\int_t^r\phi_{nt}(s,x)\,ds\right)\,dx\,dr.
\end{multline*}

Using Theorem \ref{th3.2-fabre92} and \eqref{eq3.4} we have 
\begin{equation}
    \label{eq3.6}\frac{T}{\varepsilon^3_n}\int_{1-\varepsilon_n}^1|\phi_n(t,x)|^2\,dx\leq C+\frac{2}{\varepsilon_n^3}\left|\int_0^T\int_{1-\varepsilon_n}^1\phi_{n}(t,x)\left(\int_t^r\phi_{nt}(s,x)\,ds\right)\,dx\,dr\right|.
\end{equation}

 It remains to estimate the second term of the right-hand side. Using Young's inequality with $\delta$, 
\begin{equation*}
\begin{split}
& \frac{2}{\varepsilon_n^3} \left|\int_0^T\int_{1-\varepsilon_n}^1\phi_n (t,x)\left(\int_t^r\phi_{nt} 
(s,x)\,ds\right)\,dx\,dr\right|\\
&\qquad \leq \frac{\delta }{\varepsilon_n^3}\int_{1-\varepsilon_n}^1|\phi_n(t,x)|^2\,dx+\frac{1}{4\delta\varepsilon_n^3}\int_{1-\varepsilon_n}^1\left(\int_0^T\int_t^r\phi_{nt}(s,x)\,ds\,dr\right)^2\,dx\\
&\qquad\leq  \frac{\delta}{\varepsilon_n^3}\int_{1-\varepsilon_n}^1|\phi_n(t,x)|^2\,dx+\frac{T^3}{4\delta\varepsilon_n^3}\int_{1-\varepsilon_n}^1\int_0^T|\phi_{nt}(s,x)|^2\,ds\,dx.
\end{split}
\end{equation*} 
Returning to \eqref{eq3.6}, taking $\delta=T/2$  and using \eqref{eq3.5}, we obtain \[\frac{T}{2\varepsilon_n^3}\int_{1-\varepsilon_n}^1|\phi_n (t,x)|^2\,dx\leq C+\frac{T^2}{2\varepsilon_n^3}\int_{1-\varepsilon_n}^1\int_0^T|\phi_{nt}(s,x)|^2\,dx\,ds\leq C+\frac{T^2}{2n\delta},
\]
which means that there exists $C=C(T,\alpha)>0$ such that
\begin{equation}
    \label{lema3.5}\frac{1}{\varepsilon_n^{3}} \int_{1-\varepsilon_n}^1 |\phi_n(t,x)|^2\,dx\leq C, \ \ \forall n\in\mathbb{N} \ \ \mbox{and} \ \ \forall t\in[0,T].
\end{equation}

Next, for any $n\in\mathbb{N}$ let us take $S_n\in H_\alpha^2$ satisfying $(x^\alpha S_{nx})_x=\phi_n^1$ and define 
\[
\psi_n(t,x)=\int_0^t\phi_n(s,x)\,ds+S_n.
\] 
Observe that $\psi_n$ is the solution of \eqref{2.3} with the initial data $(\psi^0_n ,\psi^1_n ):=(S_n, \phi^0_n )\in H_\alpha^2 \times H_\alpha^1$. On the other hand, recalling that \eqref{eq3.5.0} and $\phi=0$, we know that \[\phi^0_n\rightharpoonup0 \ \mbox{in} \ H_\alpha^1 \ \ \mbox{and} \ \ \phi^1_n\rightharpoonup0 \ \mbox{in} \ L^2(0,1).\] 
Since the embeddings $H_\alpha^1\hookrightarrow L^2(0,1)\hookrightarrow H_\alpha^{-1}$ are compact, passing to a subsequence if necessary, we have $\psi_{nt}(0)=\phi^0_n\to0$ in $L^2(0,1)$ and $\psi_n(0)=S_n\to0$ in $H_\alpha^1$. 
As a consequence, Proposition \ref{well-pos} implies the convergences
\begin{equation}
    \label{eq3.7}
    \psi_n \to0 \ \mbox{in} \ C^0([0,T];H_\alpha^1) \ \ \mbox{and} \ \ \psi_{nt}\to0 \ \mbox{in} \ C^0([0,T];L^2(0,1)).
\end{equation}

At this point, arguing as in the obtainment of \eqref{lema3.5},  we can conclude that \begin{equation}\label{eq3.7.0}
     \frac{1}{\varepsilon^3_n}\int_{1-\varepsilon_n}^1|\psi_n(t,x)|^2\,dx\leq C, \ \forall n\in\mathbb{N} \ \ \mbox{and} \ \ \forall t\in[0,T].
\end{equation}
   
In particular, for $t=0$, there exist an increasing sequence $(n_k)_{k=1}^{\infty}$ of positive integers and a real number $I\geq0$ such that 
\begin{equation}\label{eq3.8}
   \frac{1}{\varepsilon^3_{n_k}}\int_{1-\varepsilon_{n_k}}^1|\psi_{n_k}(0,x)|^2\,dx\to I, \emph{ as } k\to +\infty.
\end{equation}

In what follows, we will conclude that the convergence \eqref{eq3.8} actually holds for all $t\in[0,T]$. Indeed, integrating by parts, we have
\begin{multline}\label{eq3.9}
    \frac{2}{\varepsilon^3_{n_k}}\int_0^t\int_{1-\varepsilon_{n_k}}^1\psi_{n_k}(s,x)\phi_{n_k}(s,x)\,dx\,ds=\frac{1}{\varepsilon^3_{n_k}}\int_{1-\varepsilon_{n_k}}^1\int_0^t(|\psi_{n_k}|^2)_s\,ds\,dx\\ =\frac{1}{\varepsilon^3_{n_k}}\int_{1-\varepsilon_{n_k}}^1|\psi_{n_k}(t,x)|^2\,dx-\frac{1}{\varepsilon^3_{n_k}}\int_{1-\varepsilon_{n_k}}^1|\psi_{n_k}(0,x)|^2\,dx.
\end{multline}
Furthermore, from H\"older inequality\begin{multline*}
    \left|\frac{1}{\varepsilon^3_{n_k}}\int_0^t\int_{1-\varepsilon_{n_k}}^1\psi_{n_k}(s,x)\phi_{n_k}(s,x)\,dx\,ds\right|\leq \left(\frac{1}{\varepsilon^3_{n_k}}\int_0^t\int_{1-\varepsilon_{n_k}}^1|\psi_{n_k}(s,x)|^2\,dx\,ds\right)^{1/2}\\ \left(\frac{1}{\varepsilon^3_{n_k}}\int_0^t\int_{1-\varepsilon_{n_k}}^1|\phi_{n_k}(s,x)|^2\,dx\,ds\right)^{1/2}.
\end{multline*}
On the other hand, from \eqref{lema3.5}, we get 
\[
\frac{1}{\varepsilon^3_{n_k}}\int_0^t\int_{1-\varepsilon_{n_k}}^1|\phi_{n_k}(s,x)|^2\,dx\,ds\leq Ct\leq CT\leq C
\] 
and, applying Theorem \ref{th3.2-fabre92}, we have
\begin{multline}\label{eq3.9.0}
    \frac{1}{\varepsilon^3_{n_k}}\int_0^t\int_{1-\varepsilon_{n_k}}^1|\psi_{n_k}(s,x)|^2\,dx\,ds\leq \frac{1}{\varepsilon^3_{n_k}}\int_0^T\int_{1-\varepsilon_{n_k}}^1|\psi_{n_k}(s,x)|^2\,dx\,ds \\ \leq C\left(\|\psi^0_{n_k}\|^2_{H_\alpha^1}+\|\psi^1_{n_k}\|^2_{L^2(0,1)}\right)\to 0.
\end{multline}
As a consequence,
\[\frac{2}{\varepsilon^3_{n_k}}\int_0^t\int_{1-\varepsilon_{n_k}}^1\psi_{n_k}(s,x)\phi_{n_k}(s,x)\,dx\,ds\to 0, \ \forall t\in[0,T],\]
which combined with \eqref{eq3.8} and \eqref{eq3.9}, allows us to conclude that \begin{equation}
    \label{lema3.6}
   \frac{1}{\varepsilon^3_{n_k}}\int_{1-\varepsilon_{n_k}}^1|\psi_{n_k}(t,x)|^2\,dx\to I, \ \forall t\in[0,T].
\end{equation}

We claim that $I=0$. Indeed, for the sake of simplicity  consider the sequence of functions 
\[
f_k(t)=\frac{1}{\varepsilon^3_{n_k}}\int_{1-\varepsilon_{n_k}}^1|\psi_{n_k}(t,x)|^2\,dx.
\] 
Clearly, from \eqref{lema3.6} and  \eqref{eq3.7.0}, $f_k(t)\to I$, as $k\to +\infty$, and $|f_k(t)|\leq C, \ \forall k\in\mathbb{N}$, where $t\in[0,T]$. So that, using  Lebesgue's Dominated Convergence Theorem and \eqref{eq3.9.0}, we deduce that \[IT=\int_0^TI\,dt=\displaystyle\lim_{k\to +\infty} \int_0^Tf_k(t)\,dt=\lim_{k\to +\infty} \frac{1}{\varepsilon^3_{n_k}}\int_0^T\int_{1-\varepsilon_{n_k}}^1|\psi_{n_k}(s,x)|^2\,dx\,ds= 0,
\]
therefore,  $I=0$.

With all this information that we have collected about the sequence $(\psi_{n_k})$ we can turn our attention back to $(\phi_{n_k})$. 
Integrating by parts, we obtain \begin{multline*}
    \frac{1}{\varepsilon^3_{n_k}}\int_0^T\int_{1-\varepsilon_{n_k}}^1\psi_{n_k}(t,x)\phi'_{n_k}(t,x)\,dx\,dt\\ =\frac{-1}{\varepsilon^3_{n_k}}\int_0^T\int_{1-\varepsilon_{n_k}}^1|\phi_{n_k}(t,x)|^2\,dx\,dt+\frac{1}{\varepsilon^3_{n_k}}\int_{1-\varepsilon_{n_k}}^1\psi_{n_k}(\cdot,x)\phi_{n_k}(\cdot,x)|^{t=T}_{t=0}\,dx\,dt.
\end{multline*}

Using H\"older inequality, \eqref{lema3.5} and \eqref{lema3.6} we can see that the last term on the right side goes to 0. The term on the left side also goes to 0, just have in mind  \eqref{eq3.5} and \eqref{eq3.9.0}. Hence, 
\begin{equation*}
    \frac{1}{\varepsilon^3_{n_k}}\int_0^T\int_{1-\varepsilon_{n_k}}^1|\phi_{n_k}(t,x)|^2\,dx\,dt\to0
\end{equation*}As we mention in \eqref{f1}, it completes the proof.
\end{proof}

\newpage
\appendix

\renewcommand{\thesection}{\Alph{section}}

%% file: body-pages/apd1.tex
\section{Proof of the well-posedness results}\label{app-well-pos}

In this appendix, we will provide the proofs of the well-posedness results for weak and very weak solutions, namely Propositions \ref{well-pos} and \ref{well-pos-trans}. For the reader's convenience, we will restate them.

\begin{prop}\label{well-pos-app}
	Given  $f\in L^1(0,T;L^2(0,1))$ and $(y^0,y^1)\in H^1_\alpha\times L^2(0,1)$, there exists a unique weak solution $y\in C^0([0,T];H^1_\alpha)\cap C^1([0,T];L^2(0,1))$ of \eqref{pb1}.	In addition, there exists a positive constant $C=C(T,\alpha)$ such that
\begin{equation}\label{ineq1-app}
\sup_{t\in[0,T]}\left( \n{y_t(t)}{L^2(0,1)}^2+\n{y(t)}{H_\alpha^1}^2 \right) 
\leq C\left(\n{f}{L^1(0,T;L^2(0,1))}^2+\n{y^0}{H_\alpha^1}^2 +\n{y^1}{L^2(0,1)}^2\right).    
\end{equation}
\end{prop}

\begin{proof} 
 This result was established in \cite{cannarsa2015wavecontrol}, for the homogeneous case,  using a semigroup approach.  We will use the same approach to obtain a mild solution in the sense of semigroups. In the following, we will show that a mild solution is also a weak solution in the sense of Definition \ref{weak}.

    Firstly, let us consider the Hilbert space $X=L^2(0,1)\times H_\alpha^{-1}$, with the inner product
    \begin{equation*}
\left((u_1,\xi_1),(u_2,\xi_2)\right)_X =\pl{u_1}{u_2}+\iphad{\xi_1}{\xi_2},
 \end{equation*}
 where $\iphad{\cdot}{\cdot}$ is given in \eqref{inp-dual}. Also, consider $D(B)=H^1_\alpha\times L^2(0,1)$ and the unbounded linear operator $B:D(B)\subset X\to X$ given by \[B(u,v)=(v,(x^\alpha u_x)_x).\]

    A straightforward calculation shows that $(B(U),U)_X=0, \ \forall U\in D(B)$ and this leads us to conclude that $B$ and $-B$ are dissipative operators. For any  $F\in X$, from Lax-Milgran Theorem, we can prove that $U-BU=F$ has a solution in $D(B)$, which implies that $B$ is $m-$dissipative. The same is also true for $-B$. Hence, $B$ is a skew-adjoint operator. Thereofre, the semigroup theory guarantees that $B$ is the generator of an isometry group $(S(t))_{t\in\mathbb{R}}$ in $X$.

    For $U_0=(y^0,y^1)\in D(B)$, from \cite[Theorem 3.2.3]{cazenave1998}, the function $V:\mathbb{R}\to X$ given by $V(t)=S(t)U_0$ fulfills the following properties:
    \[\begin{cases}
        V\in C(\mathbb{R},D(B))\cap C^1(\mathbb{R};X);\\
        V'(t)=BV(t), \ \forall t\in\mathbb{R};\\
        \|V(t)\|_X=\|U_0\|_X \ \forall t\in\mathbb{R};\\
        V(0)=U_0 .
    \end{cases}.\]
    
    Now, for $f\in L^1(0,T;L^2(0,1))$, we have $F=(0,f)\in L^1(0,T;D(B))$. Hence, for a fixed $t\in(0,T]$, we can define $R(s)=S(t-s)F(s)$, for $0\leq s\leq t$, to obtain $R\in L^1(0,t;D(B))$. So that, the functional \[t\longmapsto\int_0^tR(s)\,ds\]lies in $W^{1,1}(0,T;D(B))\hookrightarrow C([0,T];D(B))$. From \cite[Proposition 4.1.9]{cazenave1998},  we conclude that Duhamel's Formula \[U(t)=S(t)U_0+\int_0^tS(t-s)F(s)\,ds\] defines a function in $C([0,T];D(B))$ that solves the problem \[\begin{cases}
	U\in L^1(0,T;D(B))\cap W^{1,1}(0,T;X);\\
	{\color{black}U_t=BU+F} \ \ \mbox{a.e. in } [0,T]; \\
	U(0)=U_0.
	\end{cases}\]  
    
    Finally, let us write  $U(t)=(u(t),v(t))$ and prove that $u$ is a weak solution in the sense of Definition \ref{weak}. Since $U\in C([0,T];D(B))$, it follows that $u\in C([0,T],H_\alpha^1)$ and $v\in C([0,T];L^2(0,1))$. Furthermore, it is clear that $u_t=v$, whence $u\in C^1([0,T];L^2(0,1))$, $u(0)=y^0$ and $u_t(0)=y^1$. It remains to show that \eqref{form-int2} holds for $u$. Indeed, let us take $\phi\in \mathcal{D}(0,T)$ and $w\in H_\alpha^1$, and define $\Phi=(-w\phi',w\phi)\in\mathcal{D}(0,T;D(B))$. Hence,

\[
\int_0^T\left(U_t(t),\Phi(t)\right)_X\,dt=\int_0^T\left(BU(t),\Phi(t)\right)_X \,dt+\int_0^T\left(F(t),\Phi(t)\right)_X\,dt 
\]
and this leads us to
{\color{black}
\begin{multline*}
 \int_0^T\left[-\pl{u_t}{w}\phi'+\iphad{v_t}{w}\phi\right]\,dt\\
 =\int_0^T
 \left[-\pl{v}{w}\phi'+\iphad{(x^\alpha u_x)_x}{w}\phi\right]\,dt
 + \int_0^T\iphad{f}{w}\phi\,dt.
\end{multline*}
}

Since $v\in L^1(0,T,L^2(0,1))$ and $v_t\in L^1(0,T;\had)$, we can perform an integration by parts with respect to $t$ in the second integral, achieving that $u$ verifies \eqref{form-int2}.    
\end{proof}

\begin{cor}\label{corA2}
If \(f\equiv 0\) in Proposition \ref{well-pos-app}, then \(y_{tt}\in C^0([0,T];\had) \) and there exists a positive constant $C=C(T,\alpha)$ such that
\begin{equation}\label{ineq1-app2}
\sup_{t\in[0,T]}\left( \n{y_{tt}(t)}{\had}^2\right) 
\leq C\left(\n{y^0}{H_\alpha^1}^2 +\n{y^1}{L^2(0,1)}^2\right).    
\end{equation}
\end{cor}
 
\begin{proof} {As seen in the previous proof, with \mbox{$f=0$} we have that \mbox{$U_t=BU$} a.e. in \mbox{$[0,T]$} and this lead us to \mbox{\(y_{tt}(t)=(x^\alpha y_x)_x(t)\)} a.e. in \mbox{$[0,T]$}.  Since \mbox{$y\in C^0([0,T],\ha)$}, it follows that \mbox{\(y_{tt}\in C^0([0,T],\had)\)} }.  Besides that,
\begin{equation*}
\left|\dpa{y_{tt}(t,\cdot)}{w}\right|= \left|\pl{x^{\alpha/2} y_x(t,\cdot)}{x^{\alpha/2}w_x}\right|\leq \n{y(t,\cdot)}{\ha}\n{w}{\ha}, \ \forall w\in \ha.
\end{equation*}
Therefore, inequality \eqref{ineq1-app} implies
\begin{equation*}
\sup_{t\in[0,T]}\left( \n{y_{tt}(t)}{\had}^2\right)\leq \sup_{t\in[0,T]}\left( \n{y(t)}{\ha}^2\right)\leq   C\left(\n{y^0}{H_\alpha^1}^2 +\n{y^1}{L^2(0,1)}^2\right).
\end{equation*}
\end{proof}

\begin{prop}\label{well-pos-trans-app}
	Given  $f\in L^1(0,T;L^2(0,1))$ and $(y^0,y^1)\in L^2(0,1)\times H^{-1}_\alpha$, there exists a unique solution by transposition  $y\in C^0([0,T];L^2(0,1))\cap C^1([0,T];H^{-1}_\alpha)$ of \eqref{pb1}.	In addition, there exists a positive constant $C=C(T,\alpha)$ such that
	\begin{equation}\label{ineq-2.3-app}
 \begin{split}
& 	\sup_{t\in[0,T]}\left( \n{y(t)}{L^2(0,1)}^2+\n{y_t(t)}{H_\alpha^{-1}}^2 \right) \\
& \qquad	\leq C\left(\n{f}{L^1(0,T;L^2(0,1))}^2+\n{y^0}{L^2(0,1)}^2+\n{y^1}{H_\alpha^{-1}}^2 \right).
 \end{split}
	\end{equation}
\end{prop}

\begin{proof}
    Let us fix $(y^0,y^1)\in L^2(0,1)\times H_\alpha^1$ and $f\in L^1(0,T;L^2(0,1))$. In order to save notation, we set 
\[
E_0:=\left(\|y^0\|^2_{L^2(0,1)} +\|y^1\|^2_{H_\alpha^{-1}} +\|f\|^2_{L^2(0,T;L^2(0,1))}\right)^\frac{1}{2}.
\] 
From Theorem \ref{well-pos}, we can define a continuous linear functional  $\Gamma:L^1(0,T;L^2(0,1))\to\mathbb{R}$ given by
\[
\Gamma(F)=-(y^0,\theta_t(0,\cdot))_{L^2(0,1)}+\dpa{y^1}{\theta(0,\cdot )}+\intq f\theta\dd,
\]
where $\theta\in C([0,T];H_\alpha^1)\cap C^1([0,T]; L^2(0,1))$ is the weak solution of \eqref{back-pb} with $g=F$. Furthermore, we have 
\[
\n{\Gamma}{(L^1(0,T);L^2(0,1))'}\leq CE_0.
\]
From the Riez Representation Theorem, there exists a unique $y\in L^\infty(0,T;L^2(0,1))$, with $\|y\|_{L^\infty(0,T;L^2(0,1))}=\n{\Gamma}{(L^1(0,T);L^2(0,1))'}$, such that \[\Gamma(F)=\intq yF\dd,\]
for any $F\in L^1(0,T;L^2(0,1))$. This gives us the existence and uniqueness of a very weak solution for \eqref{pb1}. It remains to show the regularity and the energy estimates. We already have
\begin{equation}\label{tp1}
\n{y}{L^\infty(0,T;L^2(0,1))}\leq CE_0.
\end{equation}
Now, we will  obtain a similar  estimate to $y_t$. To do that, let us first define
\[W_0^{1,1}(0,T;H_\alpha^{1}):=\left\{u\in L^1(0,T;\ha);\ u_t\in L^1(0,T;\ha),\ u(0,\cdot)=u(T,\cdot)=0 \right\}\]
and denote by $W^{-1,\infty}(0,T;\had)$ its dual.
Hence, for any $v\in W_0^{1,1}(0,T;H_\alpha^{1})$, we have \[|\langle y_t,v\rangle|=\left|-\int_0^T(y,v_t)_{L^2(0,T;L^2(0,1))}\,dt\right|\leq \|y\|_{L^\infty(0,T;L^2(0,1))}\|v\|_{W_0^{1,1}(0,T;H_\alpha^1)}.\]
This gives us $y_t\in W^{-1,\infty}(0,T;H_\alpha^{-1})$ and \[\|y_t\|_{W^{-1,\infty}(0,T;H_\alpha^{-1})}\leq \|y\|_{L^\infty(0,T;L^2(0,1))}\leq CE_0.\]

Next, let us prove that  \(y_t\in L^{\infty} (0,T;H_{\alpha}^{-1})\), that is, there exists \(w \in L^{\infty} (0,T;H_{\alpha}^{-1})\) such that $y_t =w\vert_{W_0^{1,1}(0,T;H_\alpha^1)}$.
For each $v\in L^1(0,T;H_\alpha^1)$, we consider a sequence $(v_n)\subset W_0^{1,1}(0,T;H_\alpha^1)$ such that $v_n\to v$ in $L^1(0,T;H_\alpha^1)$. 

{Before we continue, we need to establish the following estimate:} \begin{equation}\label{semderivada}
    \displaystyle \|\theta_{nt} (0,\cdot )\|_{L^2 (0,1)} + \|\theta_n \|_{C([0,T];H_{\alpha}^{1})} \leq  C\|v_n\|_{L^{1} (0,T;H_{\alpha}^{1})},
\end{equation}
{where \mbox{$\theta_n$} is a weak solution of \mbox{\eqref{back-pb}}, with \mbox{$g=v_{nt}$}. To do that, we will follow the steps of \mbox{\cite[Lemma 4.3]{medeiros2013introduction}}. Let \mbox{$\varphi_n$} be  the strong solution of \mbox{\eqref{back-pb}} with \mbox{$g=v_n$}. We have that \mbox{$\varphi_{nt}$} must solve \mbox{\eqref{back-pb}} with \mbox{$g=v_{nt}$}. This implies \mbox{$\varphi_{nt}=\theta_n$}. Using the energy estimates for strong solutions we can get \mbox{\eqref{semderivada}.}}

Now, \eqref{semderivada} allows us to define the function $L: L^1 (0,T;H_{\alpha}^{1}) \rightarrow \R$ given by
\[
\displaystyle L(v)=\lim_{n\to +\infty} \langle y_t,v_n\rangle_{W^{-1,\infty}(0,T;H_\alpha^{-1}),W_0^{1,1}(0,T;H_\alpha^1)}.
\]
{Indeed, if $(v_n), ~(\tilde{v}_n)  \subset W_0^{1,1}(0,T; H^1_\alpha)$
satisfies $v_n, ~\tilde{v}_n \to v$ in $L^1(0, T; H_\alpha^1 )$, then 
$v_n -\tilde{v}_n \to 0 $ em $L^1(0, T; H_\alpha^1)$. Therefore, using}
\mbox{\eqref{semderivada}, we can deduce that}
\[\begin{split}
|\langle y_t, v_n- \tilde{v}_n\rangle_{W^{-1, \infty}(0,T; H^{-1}_\alpha), W^{1,1}_0(0, T; H^1_\alpha)}|
&=\left|\int\!\!\!\!\int_{Q}y(v_{nt}-\tilde{v}_{nt})\,dx\,dt\right|=|\Gamma(v_{nt}-\tilde{v}_{nt})| \\
&\hspace{-6cm}=\left|\langle y_1, \theta_n(0,\cdot)-\tilde{\theta}_n(0,\cdot)\rangle_{H_\alpha^{-1}\times H_\alpha^1} - (y_0, \theta_{nt}(0,\cdot)-\tilde{\theta}_{nt}(0,\cdot))_{L^2(0,1)}+\int\!\!\!\!\int_{Q} f (\theta_n-\tilde{\theta}_n) \,dxdt\right|\\ 
&\hspace{-6cm}\leq \|y_1\|_{H_\alpha^{-1}}\|\theta_{n}(0,\cdot)-\tilde{\theta}_n(0,\cdot)\|_{H_\alpha^1}^2+\|y_0\|_{L^2(0,1)}.\|\theta_{nt}(0,\cdot)-\tilde{\theta}_{nt}(0,\cdot)\|_{L^2(0,1)}\\ 
&\hspace{-5.5cm} +\|f\|_{L^1(0,T;L^2(0,1))}.\|\theta_n-\tilde{\theta}_n\|_{L^\infty(0,T;L^2(0,1))}\\ &\hspace{-6cm}\leq C(\|y_1\|_{H_\alpha^{-1}}+\|y_0\|_{L^2(0,1)}+\|f\|_{L^1(0,T;L^2(0,1))})\|v_n-\tilde{v}_n\|_{L^1(0,T;H_\alpha^1)}\to 0.
\end{split}
    \]
{Hence, }
\[
\lim_{n \to \infty} \langle y_t, v_n\rangle_{W^{-1, \infty}(0,T; H^{-1}_\alpha), W^{1,1}_0(0, T; H^1_\alpha)}= \lim_{n \to \infty} \langle y_t, \tilde{v}_n\rangle_{W^{-1, \infty}(0,T; H^{-1}_\alpha), W^{1,1}_0(0, T; H^1_\alpha)}
\]
{and this means that $L$ is well defined.}

Moreover, applying \eqref{semderivada}, we see that
\begin{equation*}
    \displaystyle|\langle y_t,v_n\rangle| = |\Gamma (v_{nt})| \leq C \|v_n \|_{L^1 (0,T;H_{\alpha}^{1})} E_0,
\end{equation*}
which means that $L\in (L^{1} (0,T; H_{\alpha}^{1}))'$ and $\|L\|_{L^{\infty} (0,T;H_{\alpha}^{-1})} \leq CE_0$. Hence, there exists a unique $w\in L^{\infty} (0,T;H_{\alpha}^{-1})$ such that
\[
\displaystyle L(v)=\int_{0}^{T} \langle w(t,\cdot ), v(t,\cdot ) \rangle dt=:\langle w,v\rangle_{L^{\infty} (0,T;H_{\alpha}^{-1}),L^{1} (0,T;H_{\alpha}^{1})} 
\]
for any $v\in L^1 (0,T;H_{\alpha}^{1})$. In particular,
\[
\displaystyle \langle y_t,v\rangle_{W^{-1,\infty}(0,T;H_\alpha^{-1}),W_0^{1,1}(0,T;H_\alpha^1)}=\langle w,v\rangle_{L^{\infty} (0,T;H_{\alpha}^{-1}),L^{1} (0,T;H_{\alpha}^{1})} 
\]
for any $v\in W_0^{1,1}(0,T;H_\alpha^1)$, that is, $y_t =w|_{W_0^{1,1}(0,T;H_\alpha^1)}$. In the sequel, we will also denote $w$ by $y_t$. In this case, 
\begin{equation}\label{tp2}\|y_t\|_{L^\infty(0,T;H_\alpha^{-1})}= \|L\|_{(L^1 (0,T;H_\alpha^{1}))'}
\leq CE_0.
\end{equation}
 

Now that we have the energy estimates \eqref{tp1} and \eqref{tp2}, it remains to obtain the regularity $y\in C([0,T];L^2(0,1))\cap C^1([0,T];H_\alpha^{-1})$. 
As before, let us take a sequence $(y^0_n,y^1_n)\in H_\alpha^1\times L^2(0,1)$ such that $y^0_n\to y^0$ in $L^2(0,1)$ and $y^1_n\to y^1$ in $H_\alpha^{-1}$. Let $y_n\in C([0,T];H_\alpha^1)\cap C^1([0,T];L^2(0,1))$  be the weak solution of \eqref{pb1}, with the initial data $(y^0_n,y^1_n)$. In particular $y_n$ is also a very weak solution of \eqref{pb1}, following that $y-y_n$ is a very weak solution of \eqref{pb1} with $f=0$ and initial data $(y^0-y^0_n,y^1-y^1_n)$. From what we have already proven,
\begin{equation}\label{pb3}
\begin{split}
 \|y-y_n\|_{L^\infty(0,T;L^2(0,1))}+\|y_t-y_{nt}\|_{L^\infty(0,T;H_\alpha^{-1})}\\
\leq C\left(\|y^0-y_n^0\|_{L^2(0,1)} +\|y^1-y_n^1\|_{H_\alpha^{-1}}\right).
\end{split}
\end{equation}
This leads us to $y_n\to y$ in $L^\infty(0,T;L^2(0,1))$ and $y_{tn}\to y_t$ in $L^\infty(0,T;H_\alpha^{-1})$. 

In order to prove that $y\in C([0,T];L^2(0,1))$, we just need to see that, for any $t,t_0\in [0,T]$, we have
\begin{align*}
 & \|y(t)-y(t_0)\|_{L^2(0,1)}\\
 &\leq \|y(t)-y_{n}(t)\|_{L^2(0,1)}+\|y_n(t)-y_n(t_0)\|_{L^2(0,1)}+\|y_n(t_0)-y(t_0)\|_{L^2(0,1)}.\\
 & \leq 2\|y-y_n\|_{L^\infty(0,T;L^2(0,1))}+\|y_n(t)-y_n(t_0)\|_{L^2(0,1)}.
\end{align*}
Therefore, taking $n$ sufficiently large and using that $y_n\in C([0,T];L^2(0,1))$, we have the desired result.  Similarly, we deduce that $y_t\in C([0,T];H_\alpha^{-1})$ and finish the proof.
\end{proof}

%% file: body-pages/apd2.tex
\section{Properties of the solutions by transposition}\label{app-trans}

Given $(v^0,v^1)\in L^2(0,1)\times \had$, let $v$ be the solution by transposition (in the sense of Definition \ref{trans}) of the problem
\begin{equation}\label{pbA1}
    \begin{cases}
        v_{tt}-(x^\alpha v_x)_x=0, & (t,x)\in Q,\\
        v(t,1)=0, & t\in (0,T),\\
        \begin{cases}
            v(t,0)=0, & \text{ if } \alpha \in (0,1),\\
            \lim\limits_{x\to 0+}(x^\alpha v_x)(t,x)=0,&  \text{ if } \alpha \in [1,2)
        \end{cases}, & t\in (0,T)\\
        v(0,\cdot)=v^0, \ v_t(0,\cdot)=v^1. 
    \end{cases}
\end{equation}

\begin{prop}\label{propA1} If $v$ is the solution by transposition of \eqref{pbA1}, the following properties hold:
\begin{enumerate}[(a)]

    \item $v_{tt}-(x^\alpha v_x)_x=0$ in  $\D'(Q)$;
   \item $v_{tt}\in C^0([0,T];\habd)$;
    \item $v(0,\cdot)=v^0$ in $L^2$ and $v_t(0,\cdot )=v^1$ in $ {\had}$.
\end{enumerate}
\end{prop}

\begin{proof}
(a) Given $\theta \in \D(Q)$, we know that it is a weak solution of \eqref{back-pb}, with $g=\theta_{tt}-(x^\alpha \theta_x)_x$. In this case, 
$v$ is a solution by transposition and $\theta(0,\cdot)=\theta_t(0,\cdot)=0$, whence 
\[\intq v(\theta_{tt}-(x^\alpha \theta_x)_x)\dd=-\pl{v^0}{\theta_t(0,\cdot)}+\dpa{v^1}{\theta(0,\cdot)}=0.\]

On the other hand, since $v\in C^0(0,T;L^2(0,1))\hookrightarrow L^2(Q)$, we have $v_{tt}-(x^\alpha v_x)_x\in \D'(Q)$. Therefore,
\[\langle v_{tt}-(x^\alpha v_x)_x,\theta \rangle_{{D'(Q)},D(Q)}=\intq v\big(\theta_{tt}-(x^\alpha \theta_x)_x\big)\dd=0,\ \forall \theta \in \D(Q).\]

Additionally, observe that (b) is a consequence of (a), because  $v\in C^0([0,T];L^2(0,1))\cap C^1([0,T];\had)$, which gives us \textcolor{black}{$v_{tt}=(x^\alpha v_x)_x\in C^0([0,T];\habd)$.}

In order to prove (c), take $\xi(t,x)=\eta(t)\zeta(x)$, with $\eta\in H^2(0,T)$ satisfying $\eta(T)=\eta'(T)=0$, and $\zeta \in \haa$. Once again, since $v$ is a solution by transposition and $\xi$ is a weak solution of \eqref{back-pb}, with $g=\xi_{tt}-(x^\alpha \xi_x)_x=\eta''\zeta-\eta(x^\alpha \zeta')'$, we get
\begin{equation}\label{A1}
\intq v(\xi_{tt}-(x^\alpha \xi_x)_x)\dd=-\pl{v^0}{\zeta}\eta'(0)+\dpa{v^1}{\zeta}\eta(0).    
\end{equation}

Recalling that $v\in C^0([0,T];L^2(0,1))\cap C^1([0,T];\had)$ and  $v_{tt}\in C^0([0,T];\habd)$, we can see that

\begin{equation}\label{A2}
\begin{split}
& \intq v\xi_{tt}\dd  = \int_0^T \pl{v(t,\cdot)}{\zeta}\eta''\,dt\\
&\qquad =-\int_0^T \dpa{v_t(t, \cdot)}{\zeta}\eta'\,dt-\pl{v(0,\cdot)}{\zeta}\eta'(0)\\
&\qquad =\int_0^T \dpaa{v_{tt}(t,\cdot)}{\zeta}\eta\,dt+\dpa{v_t(0,\cdot)}{\zeta}\eta(0) -\pl{v(0,\cdot)}{\zeta}\eta'(0).
\end{split}    
\end{equation}
Also, 
\begin{equation}\label{A3}
\intq v(x^\alpha \xi_x)_x\dd = \int_0^T \dpaa{(x^\alpha v_x)_x(t,\cdot)}{\zeta}\eta\, dt. 
\end{equation}

Substituting \eqref{A2} and \eqref{A3} in \eqref{A1}, and taking into account the item (a),  we have 
\begin{equation*}
\displaystyle \begin{split}
& -\pl{v^0}{\zeta}\eta'(0)+\dpa{v^1}{\zeta}\eta(0)\\
& = \int_0^T \dpaa{v_{tt}(t,\cdot)-(x^\alpha v_x)_x(t,\cdot)}{\zeta}\eta\,dt+\dpa{v_t(0,\cdot)}{\zeta}\eta(0) -\pl{v(0,\cdot)}{\zeta}\eta'(0)\\
& =\dpa{v_t(0,\cdot)}{\zeta}\eta(0)-\pl{v(0,\cdot)}{\zeta}\eta'(0).
\end{split}
\end{equation*}
Now, we can choose $\eta(0)=1$ and $\eta'(0)=0$ to obtain $v_t(0,\cdot)=v^1$ in $\had$. Similarly, taking $\eta(0)=0$ and $\eta'(0)=1$, we get
\[\pl{v^0}{\zeta}=\pl{v(0,\cdot)}{\zeta},\ \forall \zeta \in \haa.\]
In particular, it is true for any $\zeta\in \D(0,1)$. Therefore, $v(0,\cdot)=v^0$ in $L^2(0,1)$, since $\D(0,1)$ is dense in $L^2(0,1)$.
\end{proof}

\begin{prop}[Lifting argument]\label{lifting}
 Let $v$ be the solution by transposition of \eqref{pbA1} with $(v^0,v^1)\in L^2(0,1)\times \had$ as initial data, and take $\xi \in \ha$ the solution of 
\begin{equation*}
    \begin{cases}
    (x^\alpha \xi_x)_x=v^1 &  \text{ in } (0,1),\\
       \xi(1)=0,\\
        \begin{cases}
        \xi(0)=0, & \alpha \in (0,1)\\
        \lim\limits_{x\to0^+}(x^\alpha \xi_x(x))=0, & \alpha \in [1,2),
        \end{cases}
    \end{cases}
\end{equation*}
 Then
\begin{equation*}
    V(t,x):=\int_0^t v(s,x)\,ds +\xi(x)
\end{equation*}
is a weak solution of \eqref{pbA1}, in the sense of Definition \ref{weak}, with initial data $(V_0,V_1)=(\xi,v_0)$.
\end{prop}
\begin{proof} Clearly, $V(0,x)=\xi(x)\in \ha$ and $V_t(0,x)=v(0,x)=v^0(x)\in L^2(0,1)$.

Let us see that 

\begin{equation}\label{assertion}
\displaystyle \left(x^\alpha V_x\right)_x(t)= \dps\int_0^t (x^\alpha v_x)_x(s)\, ds +v^1 \text{ in } H^{-2}_\alpha, \text{ for all } t\in [0,T].
\end{equation}
Indeed, since $v\in C^0([0,T];L^2(0,1))$, we have  
\[\left(x^\alpha v_x\right)_x,\ \left(x^\alpha V_x\right)_x\in C^0\left([0,T];\habd\right). \] 
Hence, given $z\in \haa$, we have
\begin{equation*}
\begin{split}
\dpaa{\left(x^\alpha V_x\right)_x(t)}{z}& = \pl{V(t)}{(x^\alpha z_x)_x}\\
& =\int_0^t \pl{v(s)}{(x^\alpha z_x)_x}\,ds +\pl{x^\alpha \xi_x}{z_x}\\
& =\int_0^t \dpaa{(x^\alpha v_x)_x(s)}{z}\,ds +\dpa{v^1}{z}\\
& =\dpaa{\int_0^t (x^\alpha v_x)_x(s)\,ds +v^1}{z}, \text{ for all } t\in [0,T],
\end{split}
\end{equation*}
following \eqref{assertion}.  Thus, from Proposition \ref{propA1}, we get
\begin{equation*}
\begin{split}
\left(x^\alpha V_x\right)_x(t)=\int_0^t (x^\alpha v_x)_x(s)\,ds +v^1
 = \int_0^t v_{tt}(s)\,ds +v^1 
 =v_{t}(t)=V_{tt}(t),\ \forall t\in [0,T].
\end{split}
\end{equation*}

Let us prove that $V$ is a solution by transposition. In fact, take $F\in \D(Q)$ and let  $\theta=\theta(t,x)$ be its associated solution of \eqref{back-pb}, with $g=F$.  Since $V_{tt}-(x^\alpha V_x)_x=0$ in $C^0([0,T];\habd)$ and $\theta\in C^\infty(Q)$, we have
\begin{equation*}
\begin{split}
0& =\int_0^T \dpaa{V_{tt}-(x^\alpha V_x)_x}{\theta}\,dt\\
& =\int_0^T \pl{V}{\theta_{tt}-(x^\alpha \theta_x)_x }\,dt-\dpa{V_t(0,\cdot)}{\theta(0,\cdot)}+\pl{V(0,\cdot)}{\theta_t(0,\cdot)}\\
& =\intq VF\dd -\dpa{\xi}{\theta(0,\cdot)}+\pl{v^0}{\theta_t(0,\cdot)}.
\end{split}
\end{equation*}

Since $\D(Q)$ is dense in $L^1(0,T;L^2(0,1))$, it also holds for $F\in L^1(0,T;L^2(0,1))$. Thus, $V$ is a solution by transposition of \eqref{pbA1} with initial data $(\xi,v^0)\in \ha\times L^2(0,1)$. Consequently, $V$ must be a weak solution of \eqref{pbA1}. This is due to the fact that, for the same initial data, there exists a weak solution $W\in C^0([0,T];\ha)\cap C^1([0,T];L^2(0,1))$ of \eqref{pbA1}, which is also a solution by transposition. By uniqueness, $V=W$.
\end{proof}

  
\begin{prop}\label{B3}
Let \(y\) be a weak solution of \eqref{pbA1}, in the sense of Definition \ref{weak}, with initial data 
\((y^0,y^1)\in \ha\times L^2(0,1)\). Then \(v=y_t\) is a solution by transposition of \eqref{pbA1} with the initial data 
{$(v^0,v^1):=(y^1,y_{tt}(0,\cdot))\in L^2(0,1)\times H_\alpha^{-1}$.}
\end{prop}

\begin{proof}
Firstly, we will prove that, for a more regular initial data $(y^0,y^1)\in \haa\times \ha$, \(v=y_t\) is a weak solution. 

Indeed, if $(y^0,y^1)\in \haa\times \ha$, we know that \(y\in C^0([0,T];\haa)\cap C^1([0,T];\ha)\cap C^2([0,T];L^2(0,1))\) is a strong solution of \eqref{pbA1}.  Then, we can see that \(v\in C^0([0,T];\ha)\cap C^1([0,T];L^2(0,1)) \) and has initial data
\[v^0:=v(0,\cdot)=y_t(0,\cdot)=y^1\in \ha \text{ and }  v^1:=v_t (0,\cdot)=y_{tt}(0,\cdot) \in L^2(0,1).\]
It remains to prove that $v$ satisfies identity \eqref{form-int2}. Since \(y\) is a strong solution, we know that \  \(y_{tt}=(x^\alpha y_x)_x\) a.e. in \(Q\). Multiplying by \(w\phi'\), where \(w\in \ha\) and \(\phi\in \D(0,T)\),  integrating over \(Q\) and  applying  integration by parts in the second term, we obtain
\begin{align*}
\intq v_{t}w\phi'\, dxdt& =-\intq x^\alpha y_x w_x \phi'\,dxdt\\
 & =\pl{\int_0^T -x^\alpha y_x\phi'\,dt}{w_x}.
\end{align*}
Since  \(y_t\in C^0([0,T];\ha)\), we have \(x^\alpha v_x =x^\alpha y_{tx}\in L^2(0,1)\). Hence, applying the definition of distributional derivative in the last term, we get 
\begin{align*}
\intq v_{t}w\phi'\, dxdt& =\pl{\int_0^T x^\alpha y_{tx}\phi\,dt}{w_x}\\
& =\int_0^T \pl{x^\alpha v_{x}}{w_x}\phi\,dt\\
& =\intq x^\alpha v_{x}w_x\phi\,dxdt,
\end{align*}
which is the identity \eqref{form-int2} with $f\equiv 0$, as required.

Now, since \(v\) is a weak solution, we know it is also a solution by transposition, which means that, given \(F\in L^1(0,T;L^2(0,1))\), \(v\) must satisfies
\begin{equation}\label{id-vws}
\intq vF\,dxdt =-\pl{v^0}{\theta_t(0,\cdot)}+\dpa{v^1}{\theta(0,\cdot )},
\end{equation}
where $\theta=\theta(t,x)$ is a weak solution of \eqref{back-pb} with $g=F$.

Finally, the result follows by the same density argument used in the end of Proposition \ref{norm.equiv2}. Indeed, if $(y^0,y^1)\in \ha\times L^2(0,1)$   {and $y\in C^0([0,T];H_\alpha^1)$ is the weak solution of \mbox{\eqref{pbA1}} associated with this data}, we can take a sequence $(y_n^0,y_n^1)_{n\in \mathbb{N}}\in \haa\times \ha$ such that 
\[y_n^0\to y^0 \text{ in } \ha, \ \ y_n^1\to y^1 \text{ in } L^2(0,1) \ 
  {\text{ and } y_n\to y \ \text{in} \ C^0([0,T];H_\alpha^1)\cap C^1([0,T],L^2(0,1))},\] 
  {where $y_n$ is the strong  solution of \mbox{\eqref{pbA1}} with the initial data $(y_n^0,y_n^1)$.} 
From Proposition \ref{well-pos}, we have 
\[v_n:=y_{nt}\to y_t \text{ 
 in } C^0([0,T];L^2(0,1)). \]
 From Corollary \ref{corA2} we can also see that 
\begin{equation*}
y_{ntt} \to y_{tt} \text{ in } C^0([0,T];\had).
\end{equation*}
  {If we put $v_n^0=v_n(0,\cdot)$ and $v_n^1=v_{nt}(0,\cdot)$, then 
$v_n^0=y_{nt}(0,\cdot)\to y_t(0,\cdot)=y^1$ in $L^2(0,1)$ and 
$v_n^1=y_{ntt}(0,\cdot)\to y_{tt}(0,\cdot)$ in $H_\alpha^{-1}$.
 Therefore, using the first part of this proof,  $v_n$ satisfies identity
\mbox{ \eqref{id-vws}}. In this case,  we can pass to the limit in \mbox{\eqref{id-vws}}, as $n\to 
 +\infty$, achieving that  $v=y_t$ is a solution by transposition of \mbox{\eqref{pbA1} }
 related to the initial data $(v^0,v^1):=(y^1,y_{tt}(0,\cdot))\in L^2(0,1)\times 
 H_\alpha^{-1}$.}.
\end{proof}

%% file: body-pages/apd4.tex
\section{{{Regularity of solutions by transposition with normal derivative in} }\(L^2(0,1)\)}\label{app-reg}

In Remark 3.1 of \cite{fabre1992exact,fabre1993behavior} and Remark 2.4 of  \cite{lions1988controlabilite}, the authors stated that, if \(v\) is a solution by transposition of the wave equation in $\Omega \subset \R^n$, with data $(v^0,v^1)\in L^2(\Omega)\times H^{-1}(\Omega)$, and the normal derivative $\partial_\nu v$ belongs to $L^2(\partial\Omega)$, then $(v^0,v^1)\in H_0^1(\Omega)\times L^2(\Omega)$, which means that $v$ is, in fact, a weak solution related to the initial data $(v^0,v^1)$.

The aim  of this section is to prove an analogous regularity result for the one-dimensional degenerate wave equation considered in this paper.

Fix \(\alpha\in (0,2)\). Given $(v^0,v^1)\in L^2(0,1)\times \had$, let $v$ be the solution by transposition (in the sense of Definition \ref{trans}) of the problem
\begin{equation}\label{pbC1}
    \begin{cases}
        v_{tt}-(x^\alpha v_x)_x=0, & (t,x)\in Q,\\
        v(t,1)=0, & t\in (0,T),\\
        \begin{cases}
            v(t,0)=0, & \text{ if } \alpha \in (0,1),\\
            \lim\limits_{x\to 0+}(x^\alpha v_x)(t,x)=0,&  \text{ if } \alpha \in [1,2)
        \end{cases}, & t\in (0,T)\\
        v(0,\cdot)=v^0, \ v_t(0,\cdot)=v^1. 
    \end{cases}
\end{equation}

\begin{prop}\label{reg-ND}
There exist \( \widetilde{T}_\alpha>0\) such that, for any \(T> \widetilde{T}_\alpha\) and \(v\) be the solution by transposition of \eqref{pbC1} with initial data \((v^0,v^1)\in L^2(0,1)\times \had\). If  \(v_x(\cdot,1)\in L^2(0,T)\), then \((v^0,v^1)\in \ha\times L^2(0,1)\).
\end{prop}

In order to do that, as in \cite{gueye2014exact,HSSuTe2025}, we will describe the solutions of \eqref{pbC1} using Fourier-Bessel series. This is achieved by defining the unbounded operator 
\(\A:D(\A)\subset L^2(0,1)\longrightarrow L^2(0,1)\) by
\begin{equation}\label{operA}
\begin{cases}
D(\A):=\ha\cap \haa\\
\A u=(x^\alpha u_x)_x,\ u\in D(\A),
\end{cases}
\end{equation}
which is self-adjoint, positive definite, and have a compact resolvent. By the Spectral Theorem,  there exists a Hilbertian basis \( (\phin)_{n\in \mathbb{N}^\ast}\) of \(L^2(0,1)\) and an increasing sequence \( (\lambda_n)_{n\in \mathbb{N}^\ast} \), of real numbers such that \(\lambdan>0\), \(\lambdan\to +\infty\) and
\[
\A\phin=\lambdan \phin, \ \forall n\in \mathbb{N}^\ast.
\]

As a consequence, the initial data can be represented in the form
\begin{equation}\label{init-data}
v^0=\sum_{n\in \mathbb{N}^\ast} v_n^0\,\phin, \ v^1=\sum_{n\in \mathbb{N}^\ast} v_n^1\,\phin.,
\end{equation}
where the second representation also relies on \eqref{dual-prod}. Therefore, the solution of \eqref{pbC1} is written as 
\begin{equation}\label{series-v}
v(t,x)=\sum_{n\in \mathbb{N}^\ast} v_n(t)\phin(x),
\end{equation}
where 
\begin{equation}
v_n(t)=b_ne^{i\omega_nt}+\overline{{b}_n}e^{-i\omega_nt},
\end{equation}
with
\(\omega_n:=\sqrt{\lambda_n}\) and the complex coefficients \(b_n\) given by
\begin{equation*}
b_n=\frac{1}{2}\left(v_n^0-i\frac{v_n^1}{\omega_n}\right).
\end{equation*}

Additionally, since \( (\phin)_{n\in \mathbb{N}^\ast}\) are the solutions of the following Sturm-Liouville problem
\begin{equation*}
\begin{cases}
-(x^\alpha \phin')'(x)=\lambdan \phin(x), & x\in (0,1),\\
\phin(1)=0, & \\
\begin{cases}
\phin(0)=0, & \text{ if }\alpha\in [0,1)\\
\lim\limits_{x\to 0^+}(x^\alpha \phin')=0 &   \text{ if }\alpha\in [1,2),
\end{cases}
\end{cases}
\end{equation*} 
it is given, in \cite[Proposition 3.3]{HSSuTe2025}, an explicit expressions of the spectrum. Specifically, 
\begin{equation*}
\lambdan= \kappa^2 \jn^2, \ n\in\mathbb{N}^\ast,
\end{equation*}
where 
\begin{equation*}
\kappa=\frac{2-\alpha}{2}, \ \nu=\frac{|1-\alpha|}{2-\alpha} 
\end{equation*}
and \((\jn)_{n\in \mathbb{N}^\ast}\) is the strictly increasing sequence of positive zeros of the Bessel function
\begin{equation*}
J_\nu(x)=\sum_{n\in \mathbb{N}^\ast}\frac{(-1)^n}{n!\Gamma(n+\nu+1)}\left(\frac{x}{2}\right)^{2n+\nu}, \ x\geq 0.
\end{equation*} 
And the corresponding normalized eigenfunctions are
\begin{equation*}
\phin(x)=\frac{\sqrt{2\kappa}}{\left|J_\nu^\prime(\jn)\right|}
 \sqrt{x^{1-\alpha}}J_\nu (\jn x^\kappa), \ x\in(0,1),\ n\in\mathbb{N}^\ast.
\end{equation*}

Before giving the proof of Proposition \ref{reg-ND}, we shall state two well-known results. The first result concerns about the location of the zeros of the Bessel function, whose proof can be found in \cite[Proposition 7.8]{komornik2005fourier}.

\begin{lem}\label{besselzeros}
Let \((\jn)_{n\in \mathbb{N}^\ast}\) be the positive zeros of the Bessel function \(J_\nu\). Then, the following statements hold:

\begin{enumerate}[(a)]
\item For any real number \(\nu\) given,  the zeros of the Bessel function \(J_\nu\) are simple and they form an infinite increasing sequence, \(j_{\nu,1}<j_{\nu,2}<\ldots\), tending to infinity.

\item The difference sequence   \( (j_{\nu,n+1}-\jn)_{n\in \mathbb{N}^\ast} \) converges to \(\pi\) as \(n\to +\infty\).

\item The sequence  \( (j_{\nu,n+1}-\jn)_{n\in \mathbb{N}^\ast} \) is strictly decreasing, if \(|\nu|>\frac{1}{2}\), strictly increasing, if \(|\nu|<\frac{1}{2}\), and constant, if \(|\nu|=\frac{1}{2}\).
\end{enumerate}
\end{lem}

In our case,  since 
\[
\nu=\frac{|1-\alpha|}{2-\alpha},
\]
we have that \(\nu\in \left[0,\frac{1}{2}\right]\), when \( 0\leq \alpha \leq \frac{4}{3}\) and \(\nu>\frac{1}{2}\) when \(\frac{4}{3}\leq \alpha<2\).
\begin{figure}[h]
\includegraphics[scale=0.7]{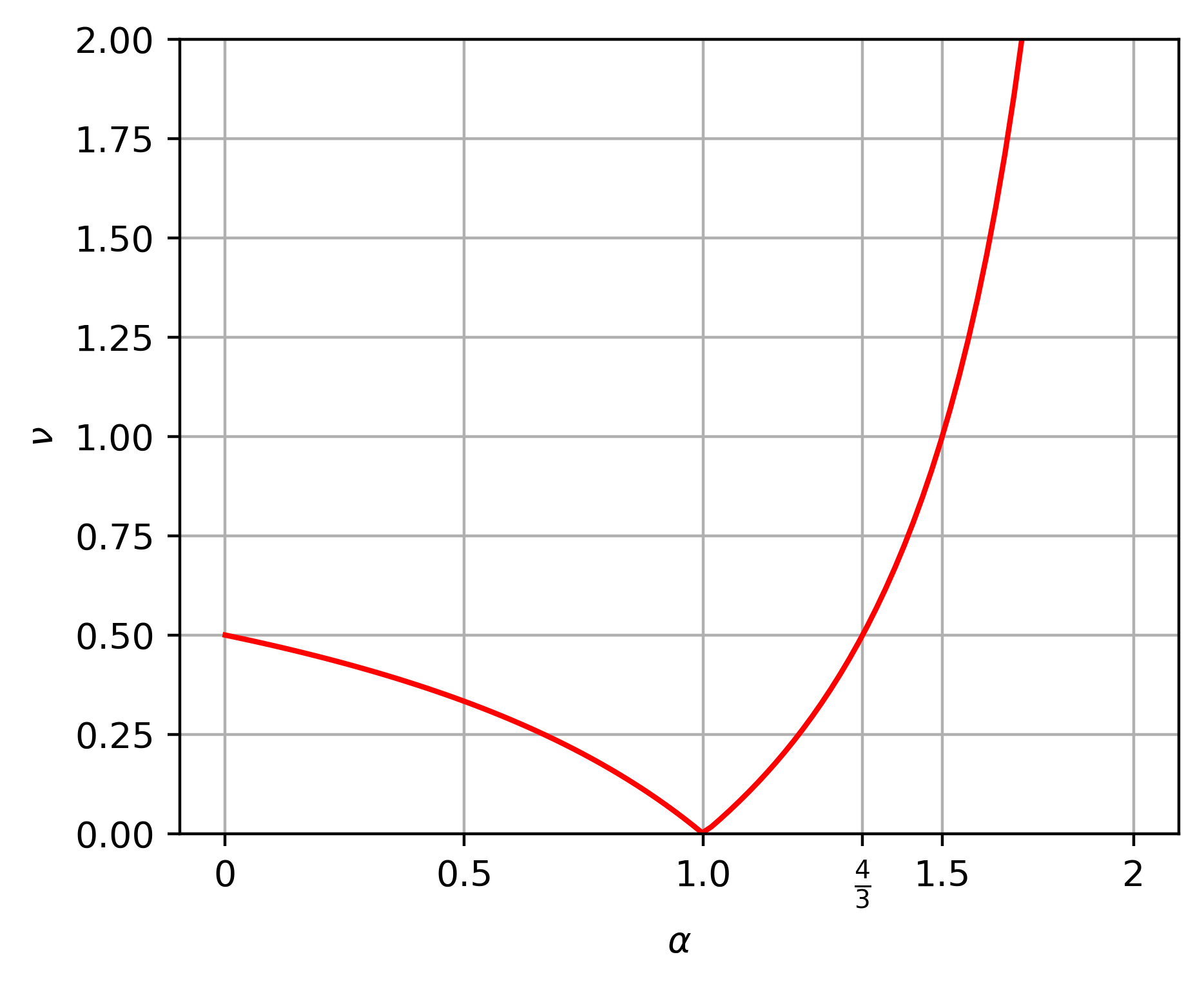}
\caption{\(\nu=\frac{|1-\alpha|}{2-\alpha}\), for \(\alpha\in (0,2)\)}
\end{figure}

Consequently, from item  (c) of Lemma \ref{besselzeros}, the following estimate holds:
\begin{equation}\label{est.jn}
(\jnn{n+1}-\jn) \geq \inf\limits_{k\in\mathbb{N}}(\jnn{k+1}-\jnn{k})
 =\begin{cases}
\jnn{2}-\jnn{1}, & 0\leq \alpha \leq \frac{4}{3},\\
\pi, & \frac{4}{3}<\alpha < 2.
\end{cases},\ \forall n\in \mathbb{N}.
\end{equation}

The second result is about Ingham inequality. Its proof is found in \cite[Theorem 4.3]{komornik2005fourier}.

\begin{lem}\label{Ingham}
Let \((\omega_k)_{k\in \mathbb{N}^\ast}\) be a sequence of real numbers that satisfies the gap condition
\begin{equation*}
\omega_{k+1}-\omega_k\geq \gamma, \ \forall n\in \mathbb{Z}.
\end{equation*}
Then, for any \(T>2\pi/\gamma\), there exists a positive constant \(C=C(T,\gamma)\) such that 
\begin{equation*}
\sum_{k\in \mathbb{Z}}|a_k|^2 \leq C \int_0^T \left|\sum_{k\in \mathbb{Z}}a_ke^{i\omega_kt}\right|^2\,dt,
\end{equation*}
for all sequence  \((a_k)_{k\in \mathbb{Z}}\in \ell^2(\mathbb{Z})\) of complex numbers.
\end{lem}

Now, we are read to prove our main result.

\begin{proof}[Proof of Proposition \ref{reg-ND} ]
Firstly, from the representation given in  \eqref{init-data}, we can see that
\begin{itemize}
\item \(
\displaystyle v^0\in \ha  \Leftrightarrow \sum_{n\in \mathbb{N}^\ast}\lambdan|v_n^0|^2<\infty
\Leftrightarrow (\sqrt{\lambdan}v_n^0 )_{n\in \mathbb{N} ^{\ast}} \in \ell^2
\)

\item 
\(
\displaystyle v^1\in L^2(0,1) \Leftrightarrow  \sum_{n\in \mathbb{N}^\ast}|v_n^1|^2<\infty
\Leftrightarrow (v_n^1 )_{n\in \mathbb{N} ^{\ast}}\in \ell^2.
\)
\end{itemize}
That is, we need to prove the convergence of that series to get the result. To do that, we will proceed as in \cite{gueye2014exact} and use Ingham Inequality.

From \eqref{series-v}, we can see that
\begin{equation*}
v_x(t,1)=\sum_{n\in \mathbb{N}^\ast}\left(b_n e^{i\omega t}+\overline{b_n}e^{-i\omega t}\right)
\phin'(1).
\end{equation*}
We can rewrite the indexes, in order to apply Ingham inequality, in the following way:
\begin{equation*}
v_x(t,1)=\sum_{k\in \mathbb{Z}}a_ke^{i{\tilde{\omega}}_kt},
\end{equation*}
where
\begin{equation*}
a_k=\begin{cases}
b_k\phin'(1), & k\geq 1\\
0, & k=0\\
\overline{b_{-k}}\phi_{-k}'(1), & k\leq -1
\end{cases} \ \text{ and } \
\tilde{\omega}_k=\begin{cases}
\omega_k, & k\geq 1\\
0, & k=0\\
-\omega_{-k}, & k\leq -1.
\end{cases} \ 
\end{equation*}

Hence, we just need to check the gap condition of Lemma \ref{Ingham}, for \((\tilde{\omega}_k)_{k\in \mathbb{Z}}\). Indeed, from \eqref{est.jn} we can see that
\begin{align}\label{gamma}
\omega_{n+1}-\omega_n\geq \gamma:= \kappa \inf\limits_{k\in\mathbb{N}}(\jnn{k+1}-\jnn{k})
=\begin{cases}
\kappa(\jnn{2}-\jnn{1}), & 0\leq \alpha \leq \frac{4}{3},\\
\kappa\pi, & \frac{4}{3}<\alpha < 2.
\end{cases},\ \forall n\in \mathbb{N}^\ast,
\end{align}
which gives the required gap condition for  \((\tilde{\omega}_k)_{k\in \mathbb{Z}}\). Therefore, Ingham inequility assures that, for any \(T> 2\pi/\gamma \),
\begin{equation*}
\sum_{n\in \mathbb{N}^\ast} |b_n\phin'(1)|^2\leq \sum_{k\in \mathbb{Z}} |a_k|^2\leq C \int_0^T|v_x(t,1)|^2\,dt<+\infty.
\end{equation*}
Now, let us to compute the left hand side of this inequality.
Since
\begin{equation*}
\phin'(1)=\sqrt{2}\jn\kappa^{3/2}\frac{J_\nu^\prime(\jn)}{\left|J_\nu^\prime(\jn)\right|},
\end{equation*}
we can see that
\begin{equation*}
|b_n\phin'(1)|^2=\frac{\kappa}{2}\left(\jn^2\kappa^2|v_n^0|^2+|v_n^1|^2\right)
=\frac{\kappa}{2}\left(\lambda_n|v_n^0|^2+|v_n^1|^2\right).
\end{equation*}
Therefore, $(\sqrt{\lambdan}v_n^0 )_{n\in \mathbb N ^{\ast}}$ and $(v_n^1\in )_{n\in \mathbb N ^{\ast}}$ belong to \(\ell^2\), as required.
\end{proof}

\begin{rem}
Ingham inequality defines \( \widetilde{T}_\alpha:=2\pi/\gamma\).  Hence, from \eqref{gamma}, we can see that
\begin{equation}\label{Taa}
 \widetilde{T}_\alpha
=\begin{cases}
\frac{4\pi}{(2-\alpha)(\jnn{2}-\jnn{1})}, & 0\leq \alpha \leq \frac{4}{3},\\
\frac{4}{2-\alpha}, & \frac{4}{3}<\alpha < 2.
\end{cases}
\end{equation}
We note that, for  \( \frac{4}{3}<\alpha<2 \), we have \(\widetilde{T}_\alpha=T_\alpha\), where $T_\alpha$ is defined in \eqref{Ta}. On the other hand, for  \(0\leq \alpha\leq \frac{4}{3}\), since \((\jnn{n+1}-\jn)\) is nondecreasing and converges to \(\pi\), we have
\[
\jnn{2}-\jnn{1} 
\leq\sup \limits_{n\in\mathbb{N}}(\jnn{n+1}-\jn)=\pi.
\]
Therefore, 
\[
 \widetilde{T}_\alpha=\left(\frac{\pi}{\jnn{2}-\jnn{1}}\right)T_\alpha\geq   T_\alpha, \text{ for any } \alpha\in \left[0,\frac{4}{3}\right].
\]
Furthermore, we can estimate how close  \( \widetilde{T}_\alpha\) is to \(T_\alpha\), when \(\alpha\in \left[0,\frac{4}{3}\right]\). In fact, since  \(0\leq \nu\leq \frac{1}{2}\), we know that \(\jnn{2}-\jnn1\) is increasing with respect to \(\nu\)  ( see \cite[Corollary 1]{lorch1964monotonicity}). Hence, because \(j_{\frac{1}{2},2}-j_{\frac{1}{2},1}=\pi\) and 
\(j_{0,2}-j_{0,1}\approx 3.11525255\), we can estimate the ratio between them as follows:
\[
1=\left(\frac{\pi}{j_{\frac{1}{2},2}-j_{\frac{1}{2},1}}\right)
\leq\frac{\widetilde{T}_\alpha}{T_\alpha} \leq \left(\frac{\pi}{j_{0,2}-j_{0,1}}\right)\approx 1.00845521.
\]
Finally, in order to visualize it better, we plot the graph of this ratio below.

\begin{figure}[h!]
\includegraphics[scale=0.6]{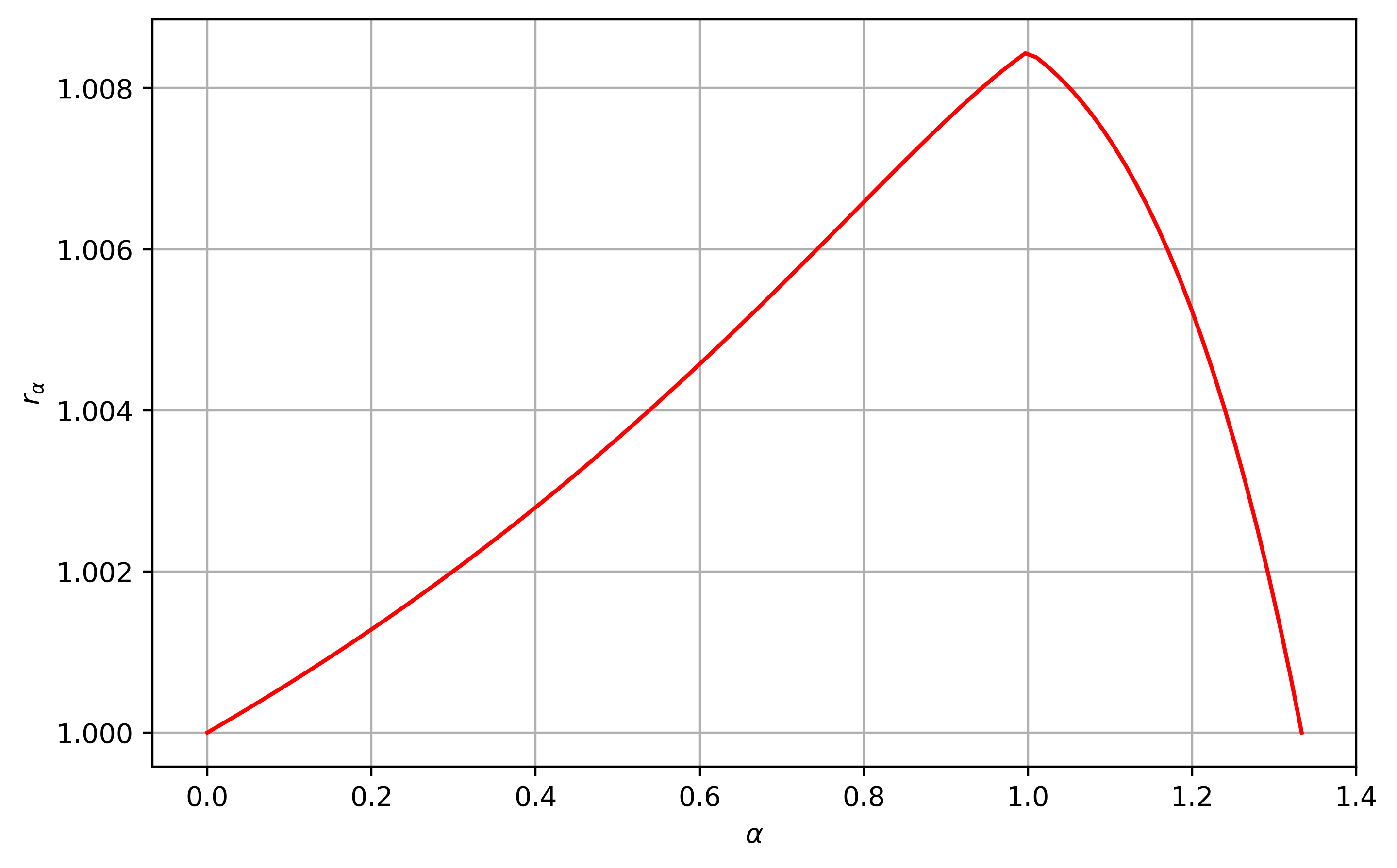}
\caption{ratio \(r_\alpha=\frac{\widetilde{T}_\alpha}{T_\alpha}\), when \(0\leq \alpha\leq \frac{4}{3}\).}
\end{figure}
\end{rem}